%% file: ritter2_212.tex
\documentclass{amsart}
\usepackage{amscd,amssymb,amsmath,amsthm}
\usepackage[dvips]{graphicx}
\usepackage[all]{xy}
\theoremstyle{plain}
\newtheorem{proposition}{Proposition}
\newtheorem{theorem}[proposition]{Theorem}
\newtheorem{definition}[proposition]{Definition}
\newtheorem{corollary}[proposition]{Corollary}
\newtheorem{lemma}[proposition]{Lemma}
\newtheorem{remark}[proposition]{Remark}

\newtheorem{proposition-definition}[proposition]{Proposition/Definition}
\newtheorem*{proposition*}{Proposition}
\newtheorem*{theorem*}{Theorem}
\newtheorem*{maintheorem*}{Main Theorem}
\newtheorem*{maincorollary*}{Main Corollary}
\newtheorem*{corollary*}{Corollary}
\newtheorem*{lemma*}{Lemma}
\newtheorem*{remark*}{Remark}
\newtheorem*{example*}{Example}

\def\co{\colon\thinspace}

\newcommand{\LN}{\mathcal{L}_0 N}

\newcommand{\N}{\mathbb{N}}
\newcommand{\Z}{\mathbb{Z}}

\newcommand{\R}{\mathbb{R}}
\newcommand{\C}{\mathbb{C}}
\newcommand{\CP}{\C\mathbb{P}^{\infty}}

\newcommand{\supp}{\textrm{supp}\,}

\begin{document}

\title[Deformations of symplectic cohomology]{Deformations of symplectic cohomology\\ and exact Lagrangians in ALE spaces}

\author{Alexander F. Ritter}

\address{Department of Mathematics, M.I.T., Cambridge, MA 02139, USA.}

\email{ritter@math.mit.edu}


\begin{abstract}
ALE spaces are the simply connected hyperk\"{a}hler manifolds
which at infinity look like $\C^2/G$, for any finite subgroup
$G\subset SL_2(\C)$. We prove that all exact Lagrangians inside
ALE spaces must be spheres. The proof relies on showing the
vanishing of a twisted version of symplectic cohomology.

This application is a consequence of a general deformation
technique. We construct the symplectic cohomology for non-exact
symplectic manifolds, and we prove that if the non-exact
symplectic form is sufficiently close to an exact one then the
symplectic cohomology coincides with an appropriately twisted
version of the symplectic cohomology for the exact form.
\end{abstract}
\maketitle
%
%
\section{Introduction}
An \emph{ALE hyperk\"{a}hler} manifold $M$ is a non-compact
simply-connected hyperk\"{a}hler $4-$manifold which asymptotically
looks like the standard Euclidean quotient $\C^2/\Gamma$ by a
finite subgroup $\Gamma \subset SU(2)$. These spaces can be
explicitly described and classified by a hyperk\"{a}hler quotient
construction due to Kronheimer \cite{Kronheimer}.

ALE spaces have been studied in a variety of contexts. In
theoretical physics they arise as gravitational instantons in the
work of Gibbons and Hawking. In singularity theory they arise as
the minimal resolution of the simple singularity $\C^2/\Gamma$. In
symplectic geometry they arise as plumbings of cotangent bundles
$T^*\C P^1$ according to ADE Dynkin diagrams:\\[1mm]

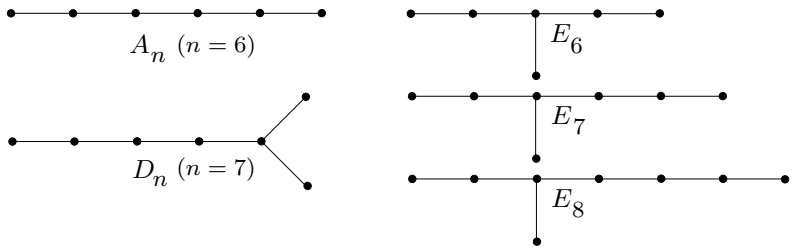
\begin{figure}[ht] \centering \scalebox{1.1}{
\input{a_dynkin.pstex_t} }
\end{figure}

Recall that the finite subgroups $\Gamma\subset SU(2)$ are the
preimages under the double cover $SU(2)\to SO(3)$ of the cyclic
group $\Z_n$, the dihedral group $\mathbb{D}_{2n}$, or one of the
groups $\mathbb{T}_{12}$, $\mathbb{O}_{24}, \mathbb{I}_{60}$ of
rigid motions of the Platonic solids. These choices of $\Gamma$
will make $\C^2/\Gamma$ respectively a singularity of type
$A_{n-1}$, $D_{n+2}$, $E_6$, $E_7$, $E_8$. The singularity is
described as follows. The $\Gamma-$invariant complex polynomials
in two variables are generated by three polynomials $x,y,z$ which
satisfy precisely one polynomial relation $f(x,y,z)=0$. The
hypersurface $\{f=0\}\subset \C^3$ has precisely one singularity
at the origin. The minimal resolution of this singularity over the
singular point $0$ is a connected union of copies of $\C P^1$ with
self-intersection $-2$, which intersect each other transversely
according to the corresponding $ADE$ Dynkin diagram. Each vertex
of the diagram corresponds to a $\C P^1$ and an edge between $C_i$
and $C_j$ means that $C_i \cdot C_j = 1$. We suggest Slodowy
\cite{Slodowy} or Arnol'd \cite{Arnold} for a survey of this
construction.

In the symplectic world these spaces can be described as the
plumbing of copies of $T^*\C P^1$ according to ADE Dynkin
diagrams. Each vertex of the Dynkin diagram corresponds to a disc
cotangent bundle $DT^*\C P^1$ and each edge of the Dynkin diagram
corresponds to identifying the fibre directions of one bundle with
the base directions of the other bundle over a small patch, and
vice-versa. The boundary can be arranged to be a standard contact
$S^3/\Gamma$, and along this boundary we attach an infinite
symplectic cone $S^3/\Gamma \times [1,\infty)$ to form $M$ as an
exact symplectic manifold.

Any hyperk\"{a}hler manifold comes with three canonical symplectic
forms $\omega_I, \omega_J, \omega_K$ which give rise to an
$S^2-$worth of symplectic forms: $\omega_u = u_I \omega_I + u_J
\omega_J + u_K \omega_K$, where $u=(u_I,u_J,u_K)\in S^2 \subset
\R^3$.

An ALE space is an exact symplectic manifold with respect to
$\omega_J$, $\omega_K$ or any non-zero combination $u_J\omega_J +
u_K \omega_K$. Let $d\theta$ be one of these forms. The copies of
$\C P^1$ described above are exact Lagrangian submanifolds with
respect to $d\theta$, and a neighbourhood of this chain of $\C
P^1$'s is symplectomorphic to the plumbing of $T^*\C P^1$'s by
Weinstein's Lagrangian neighbourhood theorem. The ALE space is not
exact for $\omega_u$ if $u_I \neq 0$, in which case the $\C P^1$'s
are symplectic submanifolds.
\\[2mm]
\textbf{Question.} \emph{What are the exact Lagrangian
submanifolds inside an ALE space?}

Recall that a submanifold $j:L^{n} \hookrightarrow M^{2n}$ inside
an exact symplectic manifold $(M,d\theta)$ is called \emph{exact
Lagrangian} if $j^*\theta$ is exact.\\[2mm]
\indent For example, the $A_1-$plumbing is $M=T^*S^2$ and the
graph of any exact $1-$form on $S^2$ is an exact Lagrangian sphere
in $T^*S^2$. Viterbo \cite{Viterbo3} proved that there are no
exact tori in $T^*S^2$. For homological reasons, the only
orientable exact Lagrangians in $T^*S^2$ are spheres, and we
proved in \cite{Ritter} that $L$ cannot be unorientable. Moreover,
for exact spheres $L\subset T^*S^2$, it is known that $L$ is
isotopic to the zero section (Eliashberg-Polterovich
\cite{Eliashberg}), indeed it is Hamiltonian isotopic (Hind
\cite{Hind}). Thus the only exact Lagrangians in $T^*S^2$ are
spheres isotopic to the zero section.

\begin{theorem*}
The only exact Lagrangians inside the ALE space $(M,d\theta)$ are
spheres, in particular there are no unorientable exact
Lagrangians. For example, this holds for the plumbing of copies of
$T^*\C P^1$ as prescribed by an ADE Dynkin diagram.
\end{theorem*}

We approach this problem via \emph{symplectic cohomology}, which
is an invariant of symplectic manifolds with contact type
boundary. It is constructed as a direct limit of Floer cohomology
groups for Hamiltonians which become steep near the boundary.
Symplectic cohomology can be thought of as an obstruction to the
existence of exact Lagrangians in the following sense.

Viterbo \cite{Viterbo1} proved that an exact $j:L\hookrightarrow
(M,d\theta)$ yields a commutative diagram
$$
\xymatrix{ H_{n-*}(\mathcal{L}L) \cong SH^*(T^*L,d\theta)
\ar@{<-_{)}}^{c_*}@<-5ex>[d] \ar@{<-}[r]^-{SH^*(j)} &
SH^*(M,d\theta) \ar@{<-}^{c_*}@<1ex>[d] \\
H_{n-*}(L)\cong H^*(L) \ar@{<-}[r]^-{j^*} & H^*(M) }
$$
where $\mathcal{L}L = C^{\infty}(S^1,L)$ is the space of free
loops in $L$ and the left vertical map is induced by the inclusion
of constants $c:L \to \mathcal{L}L$. The element $c_*(j^*1)$
cannot vanish, and thus the vanishing of $SH^*(M,d\theta)$ would
contradict the existence of $L$.

It is possible to describe the ALE space $(M,\omega_u)$ as a
symplectic manifold with contact type boundary with a
semi-infinite collar attached along the boundary, so that
$SH^*(M,\omega_u)$ is well-defined. The symplectic cohomology
$SH^*(M,d\theta)$ is never zero, indeed it contains a copy of the
ordinary cohomology $H^*(M)\hookrightarrow SH^*(M,d\theta)$.
However, we will show that if we make a generic infinitesimal
perturbation of the closed form $d\theta$, then the symplectic
cohomology will vanish. From this it will be easy to deduce that
the only exact Lagrangians $L\subset M$ must be spheres.

We constructed the infinitesimally perturbed symplectic cohomology
in \cite{Ritter} as follows. For any $\alpha \in H^1(\LN)$, we
constructed the associated Novikov homology theory for
$SH^*(M,d\theta)$. This involves introducing a local system of
coefficients $\underline{\Lambda}_{\alpha}$ taking values in the
ring of formal Laurent series $\Lambda=\Z (\!( t )\!)$. Let's
denote this \emph{twisted symplectic cohomology} by
$SH^*(M,d\theta;\underline{\Lambda}_{\alpha})$.

We proved that the above functoriality diagram holds in this
context -- with the understanding that for unorientable $L$ we use
$\Z_2=\Z/2\Z$ coefficients and the Novikov ring $\Z_2 (\!( t )\!)$
instead.

Consider a transgressed form $\alpha = \tau \beta$, where $\tau :
H^2(M) \to H^1(\mathcal{L}M)$ is the transgression. The
functoriality diagram simplifies to
$$
\xymatrix{ H_{n-*}(\mathcal{L}L;\underline{\Lambda}_{\tau
j^*\beta}) \ar@{<-}^{c_*}@<-5ex>[d] \ar@{<-}[r]^-{SH^*(j)} &
SH^*(M,d\theta;\underline{\Lambda}_{\tau\beta}) \ar@{<-}^{c_*}@<1ex>[d] \\
H_{n-*}(L)\otimes \Lambda \cong H^*(L)\otimes \Lambda
\ar@{<-}[r]^-{j^*} & H^*(M)\otimes \Lambda }
$$

For surfaces $L$ which aren't spheres the transgression vanishes,
so $H_*(\mathcal{L}L;\underline{\Lambda}_{\tau j^*\beta})$
simplifies to $H_*(\mathcal{L}L) \otimes \Lambda$ and the left
vertical arrow $c_*$ becomes injective. Thus $c_*(j^*1)$ cannot
vanish, which contradicts the commutativity of the diagram if we
can show that $SH^*(M,d\theta;\underline{\Lambda}_{\tau\beta})=0$
for some $\beta$.

\begin{theorem*}
Let $M$ be an ALE space. Then for generic $\beta$,
$$SH^*(M,d\theta;\underline{\Lambda}_{\tau\beta})=0.$$
\end{theorem*}

It turns out that there is a way to prove that the non-exact
symplectic cohomology $SH^*(M,\omega)$ vanishes for a generic form
$\omega$. So to prove the above vanishing result, we need to
relate the twisted symplectic cohomology to the non-exact
symplectic cohomology. We will prove the following general result.
\begin{theorem*}
Let $(M,d\theta)$ be an exact symplectic manifold with contact
type boundary and let $\beta$ be a closed two-form compactly
supported in the interior of $M$. Then, at least for $\|\beta\|<
1$, there is an isomorphism
$$SH^*(M,d\theta+\beta)
\to SH^*(M,d\theta;\underline{\Lambda}_{\tau\beta}).$$
\end{theorem*}

For our ALE space $M$, we actually show that this result applies
to a large non-compact deformation from $d\theta$ to the non-exact
symplectic form $\omega_I$ which has a lot of symmetry. This
symmetry will be the key to proving the vanishing of
$SH^*(M,\omega_I)$ and therefore the vanishing of
$SH^*(M,d\theta;\underline{\Lambda}_{\tau\omega_I})=0$, which
concludes the proof of the non-existence of exact Lagrangians
which aren't spheres.

The key to the vanishing of $SH^*(M,\omega_I)$ is the existence of
a global Hamiltonian $S^1-$action, which at infinity looks like
the action $(a,b) \mapsto (e^{2\pi i t} a,e^{2\pi i t} b)$ on
$\C^2/\Gamma$. We will show that the grading of the $1-$periodic
orbits grows to negative infinity when we accelerate this
Hamiltonian $S^1-$action, and this will imply that
$SH^*(M,\omega_I)=0$ because a generator would have to have
arbitrarily negative grading. This concludes the argument.

The hyperk\"{a}hler construction of $M$ depends on certain
parameters, and the cohomology class of $\omega_I$ varies linearly
with these parameters. Indeed, it turns out that the form
$\omega_I$ can be chosen to represent a generic class in
$H^2(M;\R)$.

\begin{theorem*}
Let $M$ be an ALE space. Given a generic class in $H^2(M;\R)$, it
is possible to choose a symplectic form on $M$ representing this
class such that
$$SH^*(M,\omega)=0.$$
\end{theorem*}

The outline of the paper is as follows. In section \ref{Section
Symplectic manifolds with contact boundary} we recall the basic
terminology of symplectic manifolds with contact type boundary and
we define the moduli spaces used to define symplectic cohomology.
In section \ref{Section Symplectic cohomology} we construct the
symplectic cohomology $SH^*(M,\omega)$ for a (possibly non-exact)
symplectic form $\omega$, in particular in \ref{Subsection
Novikov-symplectic chain complex} we define the underlying Novikov
ring $\Lambda$ that we use throughout. In section \ref{Section
Twisted symplectic cohomology} we define the twisted symplectic
cohomology $SH^*(M,d\theta;\underline{\Lambda}_{\alpha})$, in
particular the Novikov bundle $\underline{\Lambda}_{\alpha}$ is
defined in \ref{SubsectionNovikovbundles} and the functoriality
property is described in \ref{Subsection Twisted Viterbo
Functoriality}. In section \ref{Section Grading of symplectic
cohomology} we define the grading on symplectic cohomology, which
is a $\Z-$grading if $c_1(M)=0$. In section \ref{Section
Deformation of the Symplectic cohomology} we prove the deformation
theorem which relates the twisted symplectic cohomology to the
non-exact symplectic cohomology. In section \ref{Section ADE
spaces} we recall Kronheimer's hyperk\"{a}hler quotient
construction of ALE spaces,
and we describe the details of the proof outlined above.\\[2mm]
\noindent \textbf{Acknowledgements:} I would like to thank Paul
Seidel for suggesting this project.
%
%
%
%
\section{Symplectic manifolds with contact boundary}
\label{Section Symplectic manifolds with contact boundary}
%
\subsection{Symplectic manifolds with contact type boundary}
\label{Subsection Symplectic manifolds with contact type boundary}
%
Let $(M^{2n},\omega)$ be a compact symplectic manifold with
boundary. The contact type boundary condition requires that there
is a Liouville vector field $Z$ defined near the boundary
$\partial M$ which is strictly outwards pointing along $\partial
M$. The Liouville condition is that near the boundary $\omega =
d\theta$, where $\theta=i_Z \omega$. This definition is equivalent
to requiring that $\alpha = \theta|_{\partial M}$ is a contact
form on $\partial M$, that is $d\alpha = \omega|_{\partial M}$ and
$\alpha \wedge (d\alpha)^{n-1} > 0$ with respect to the boundary
orientation on $\partial M$.

The Liouville flow of $Z$ is defined for small negative times $r$,
and it parametrizes a collar $(-\epsilon,0]\times
\partial M$ of $\partial M$ inside $M$. So we may glue an infinite
symplectic cone $([0,\infty)\times \partial M,d(e^r\alpha))$ onto
$M$ along $\partial M$, so that $Z$ extends to $Z=\partial_r$ on
the cone. This defines the completion $\widehat{M}$ of $M$,
$$
\widehat{M}=M \cup_{\partial M} [0,\infty)\times \partial M.
$$
We call $(-\epsilon,\infty)\times \partial M$ the collar of
$\widehat{M}$. We extend $\theta$ and $\omega$ to the entire
collar by $\theta=e^r \alpha$ and $\omega=d\theta$.

Let $J$ be an $\omega-$compatible almost complex structure on
$\widehat{M}$ and denote by $g=\omega(\cdot,J\cdot)$ the
$J-$invariant metric. We always assume that $J$ is of
\emph{contact type} on the collar, that is $J^*\theta=e^r dr$ or
equivalently $J\partial_r = \mathcal{R}$ where $\mathcal{R}$ is
the Reeb vector field. This implies that $J$ restricts to an
almost complex structure on the contact distribution $\ker
\alpha$. We will only need the contact type condition for $J$ to
hold for $e^r \gg 0$ so that a certain maximum principle applies
there.

From now on, we make the change of coordinates $R=e^r$ on the
collar so that, redefining $\epsilon$, the collar will be
parametrized as the tubular neighbourhood $(\epsilon,\infty)
\times \partial M$ of $\partial M$ in $\widehat{M}$, so that the
contact hypersurface $\partial M$ corresponds to $\{R = 1\}$.

In the exact setup, that is when $\omega=d\theta$ on all of $M$,
we call $(M,d\theta)$ a \emph{Liouville domain}. In this case $Z$
is defined on all of $\widehat{M}$ by $i_Z\omega=\theta$, and
$\widehat{M}$ is the union of the infinite symplectic collar
$((-\infty,\infty)\times
\partial M,d(R\alpha))$ and the zero set of $Z$.
%
\subsection{Reeb and Hamiltonian
dynamics}\label{Subsection Reeb Dynamics}
The Reeb vector field $\mathcal{R} \in C^{\infty}(T \partial M)$
on $\partial M$ is defined by $i_{\mathcal{R}} d\alpha = 0$ and
$\alpha(\mathcal{R})=1$. The periods of the Reeb vector field form
a countable closed subset of $[0,\infty)$, provided we choose
$\alpha$ generically.

For $H \in C^{\infty}(\widehat{M},\R)$ we define the Hamiltonian
vector field $X_H$ by
$$
\omega(\cdot,X_H) = dH.
$$
If inside $M$ the Hamiltonian $H$ is a $C^2$-small generic
perturbation of a constant, then the $1$-periodic orbits of $X_H$
inside $M$ are constants corresponding precisely to the critical
points of $H$.

Suppose $H=h(R)$ depends only on $R=e^r$ on the collar. Then $X_H=
h'(R) \mathcal{R}$. It follows that every non-constant
$1$-periodic orbit $x(t)$ of $X_H$ which intersects the collar
must lie in $\{ R \}\times \partial M$ for some $R$ and must
correspond to a Reeb orbit $z(t) = x(t/T):[0,T] \to
\partial M$ with period $T = h'(R)$. Since the Reeb periods
are countable, if we choose $h$ to have a generic constant slope
$h'(R)$ for $R \gg 0$ then there will be no $1$-periodic orbits of
$X_H$ outside of a compact set of $\widehat{M}$.
%
\subsection{Action 1-form}\label{Subsection Action 1-form}
%
Let $\mathcal{L}\widehat{M} = C^{\infty}(S^1,\widehat{M})$ be the
space of free loops in $\widehat{M}$. Suppose for a moment that
$\omega=d\theta$ were exact on all of $\widehat{M}$, then one
could define the $H-$perturbed action functional for $x\in
\mathcal{L}\widehat{M}$ by
$$
A_H(x) = - \int x^*\theta + \int_0^1 H(x(t)) \, dt.
$$
If $H=h(R)$ on the collar then this reduces to $A_H(x)=-R
h'(R)+h(R)$ where $x$ is a $1$-periodic orbit of $X_H$ in $\{ R \}
\times \partial M$.
The differential of $A_H$ at $x\in \mathcal{L}\widehat{M}$ in the
direction $\xi \in T_x\mathcal{L}\widehat{M} =
C^{\infty}(S^1,x^*T\widehat{M})$ is
$$
dA_H \cdot \xi = - \int_0^1 \omega(\xi, \dot{x} - X_H) \, dt.
$$

In the case when $\omega$ is not exact on all of $\widehat{M}$,
$A_H$ is no longer well-defined, however the formula for $dA_H$
still gives a well-defined $1-$form on $\mathcal{L}\widehat{M}$.
The zeros $x$ of $dA_H$ are precisely the $1$-periodic Hamiltonian
orbits $\dot{x}(t)=X_H(x(t))$.

It also meaningful to say how $A_H$ varies along a smooth path $u$
in $\mathcal{L}\widehat{M}$ by defining
$$
\partial_s A_H(u) = dA_H\cdot \partial_s u,
$$
but the total variation $\int \partial_s A_H(u)\, ds$ will depend
on $u$, not just on the ends of $u$.
%
%
%
\subsection{Floer's equation}\label{Subsection Floers Equation}
With respect to the $L^2-$metric $\int_0^1 g(\cdot,\cdot) \, dt$
the gradient corresponding to $dA_H$ is $\nabla A_H = J(\dot{x} -
X_H)$. For $u: \R \times S^1 \to M$, the negative $L^2-$gradient
flow equation $\partial_s u = -\nabla A_H(u)$ in the coordinates
$(s,t) \in \R \times S^1$ is
$$
\partial_s u + J(\partial_t u - X_H) = 0 \quad \textrm{(Floer's equation)}.
$$
Let $\mathcal{M}'(x_{-},x_{+})$ denote the moduli space of
solutions $u$ to Floer's equation, which at the ends converge
uniformly in $t$ to the $1$-periodic orbits $x_{\pm}$:
$$
\lim_{s \to \pm \infty} u(s,t) = x_{\pm}(t).
$$
These solutions $u$ occur in $\R-$families because we may
reparametrize the $\R$ coordinate by adding a constant. Denote the
quotient by $\mathcal{M}(x_{-},x_{+}) = \mathcal{M}'(x_{-},x_{+})/
\R$. To emphasize the context, we may also write
$\mathcal{M}^H(x_-,x_+)$ or $\mathcal{M}(x_{-},x_{+};\omega)$.

The action $A_H$ decreases along $u$ since
$$
\partial_s(A_H(u)) = dA_H \cdot \partial_s u = - \int_0^1 \omega(\partial_s u,
\partial_t u - X_H) \, dt = - \int_0^1 |\partial_s u|_g^2 \, dt \leq 0.
$$
If $\omega$ is exact on $M$, the action decreases by
$A_H(x_-)-A_H(x_+)$ independently of the choice of $u \in
\mathcal{M}(x_-,x_+)$.
%
%
\subsection{Energy}\label{Subsection Energy}
For a Floer solution $u$ the energy is defined as
$$
\begin{array}{ll}
E(u) & = \int |\partial_s u|^2 \, ds \wedge dt = \int
\omega(\partial_s u,
\partial_t u - X_H)\, ds \wedge dt \\
& = \int u^*\omega + \int H(x_-) \, dt - \int H(x_+) \, dt.
\end{array}
$$
If $\omega$ is exact on $M$ then for $u \in
\mathcal{M}(x_{-},x_{+})$ there is an a priori energy estimate,
$E(u) = A_H(x_{-}) - A_H(x_{+})$.

%
%
\subsection{Transversality and
compactness}\label{Subsection Transversality and Compactness}
%
Standard Floer theory methods can be applied to show that for a
generic time-dependent perturbation $(H_t,J_t)$ of $(H,J)$ there
are only finitely many $1-$periodic Hamiltonian orbits and the
moduli spaces $\mathcal{M}(x_{-},x_{+})$ are smooth manifolds. We
write
$\mathcal{M}_k(x_{-},x_{+})=\mathcal{M}_{k+1}'(x_{-},x_{+})/\R$
for the $k$-dimensional part of $\mathcal{M}(x_{-},x_{+})$.

As explained in detail in Viterbo \cite{Viterbo1} and Seidel
\cite{Seidel}, there is a maximum principle which prevents Floer
trajectories $u \in \mathcal{M}(x_{-},x_{+})$ from escaping to
infinity.

\begin{lemma}[Maximum principle]\label{Lemma Maximum principle}
If on the collar $H=h(R)$ and $J$ is of contact type, then for any
local Floer solution $u: \Omega \to [1,\infty) \times \partial M$
defined on a compact $\Omega \subset \R \times S^1$, the maxima of
$R \circ u$ are attained on $\partial \Omega$. If $H_s=h_s(R)$ and
$J=J_s$ depend on $s$, the result continues to hold provided that
$\partial_s h_s'\leq 0$. In particular, Floer solutions of
$\partial_s u + J(\partial_t u - X_H)=0$ or $\partial_s u +
J_s(\partial_t u - X_{H_s})=0$ converging to $x_{\pm}$ at the ends
are entirely contained in the region $R\leq \max R(x_{\pm})$.
\end{lemma}
\begin{proof}
On the collar $u(s,t)=(R(s,t),m(s,t)) \in [1,\infty) \times
\partial M$ and we can orthogonally decompose
$$
T([1,\infty) \times \partial M) = \R \partial_r \oplus \R
\mathcal{R} \oplus \xi
$$
where $\xi=\ker \alpha$ is the contact distribution. By the
contact type condition, $J\partial_r=\mathcal{R}$, $J\mathcal{R} =
-\partial_r$ and $J$ restricts to an endomorphism of $\xi$. Since
$X_H=h'(R)\mathcal{R}$, Floer's equation in the first two summands
$\R \partial_r \oplus \R \mathcal{R}$ after rescaling by $R$ is
$$
\partial_s R - \theta(\partial_t u) + Rh' =0 \qquad\qquad
\partial_t R + \theta(\partial_s u) = 0.
$$
Adding $\partial_s$ of the first and $\partial_t$ of the second
equation, yields $\partial_s^2 R + \partial_t^2 R + R\partial_s h'
=|\partial_s u|^2$. So $LR\geq 0$ for the elliptic operator
$L=\partial_s^2 + \partial_t^2 + R h''(R)\partial_s$, thus a
standard result in PDE theory \cite[Theorem 6.4.4]{Evans} ensures
the maximum principle for $R\circ u$.

If $h_s$ depends on $s$ and $\partial_s h_s' \leq 0$, then we get
$LR =|\partial_s u|^2-R(\partial_s h_s')(R) \geq 0$ which
guarantees the maximum principle for $R$.
\end{proof}

\begin{figure}[ht] \centering \scalebox{0.5}{
\input{compactness.pstex_t} }
\caption{The $x,y',y'',y$ are $1$-periodic orbits of $X_H$, the
lines are Floer solutions in $M$. The $u_n \in \mathcal{M}_1(x,y)$
are converging to the broken trajectory $(u_1',u_2') \in
\mathcal{M}_0(x,y')\times \mathcal{M}_0(y',y)$.}\label{Figure
Compactness}
\end{figure}
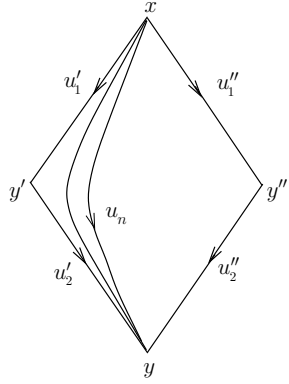

If $\omega$ were exact on $M$, then the a priori energy estimate
for $\mathcal{M}(x_{-},x_{+})$ described in \ref{Subsection
Energy} together with the maximum principle would ensure that the
moduli spaces $\mathcal{M}(x_{-},x_{+})$ have compactifications
$\overline{\mathcal{M}}(x_{-},x_{+})$ whose boundaries are defined
in terms of broken Floer trajectories (Figure \ref{Figure
Compactness}). In the proof of compactness, the exactness of
$\omega$ excludes the possibility of bubbling-off of
$J-$holomorphic spheres.

In the non-exact case if we assume that no bubbling-off of
$J$-holomorphic spheres occurs, then the same techniques guarantee
that, for any $K \in \R$, the moduli space
$$
\mathcal{M}(x_-,x_+;K) = \{ u\in \mathcal{M}(x_-,x_+): E(u)\leq K
\}
$$
of bounded energy solutions has a compactification by broken
trajectories.

\textbf{Assumptions.} \emph{We assume henceforth that no bubbling
occurs. If $c_1(M)=0$ this will hold by Hofer-Salamon
\cite{Hofer-Salamon}, as in our applications. To keep the notation
under control we continue to write $(H,J)$ although one should use
perturbed $(H_t,J_t)$.}
%
%
%
%
\section{Symplectic cohomology}
\label{Section Symplectic cohomology}
%
\subsection{Novikov symplectic chain complex}\label{Subsection
Novikov-symplectic chain complex}
%
%

Let $\Lambda$ denote the Novikov ring,
$$\Lambda = \left\{ \sum_{j=0}^{\infty} n_{j} t^{a_j} :
n_{j} \in \Z, a_j \in \R, \lim_{j \to \infty} a_j = \infty
\right\}.$$
In \cite{Ritter} we allowed only integer values of $a_j$ because
we were always using integral forms. In that setup $\Lambda$ was
just the ring of formal integral Laurent series. In the present
paper the $a_j$ will arise from integrating real forms so we use
real $a_j$.

For an abelian group $G$ the Novikov completion $G(\!(t)\!)$ is
the $\Lambda-$module of formal sums $\sum_{j=0}^{\infty} g_j
t^{a_j}$ where $g_j \in G$ and the real numbers $a_j \to \infty$.

Let $H \in C^{\infty}(\widehat{M},\R)$ be a Hamiltonian which on
the collar is of the form $H=h(R)$, where $h$ is linear at
infinity. Define $CF^*$ to be the abelian group freely generated
by $1$-periodic orbits of $X_H$,
$$
CF^*(H) =\bigoplus \left\{ \Z x : x \in \mathcal{L}\widehat{M},\;
\dot{x}(t) = X_H(x(t)) \right\}.
$$
It is always understood that we first make a generic $C^2-$small
time-perturbation $H_t$ of $H$, so that there are only finitely
many $1-$periodic orbits of $X_{H_t}$ and therefore $CF^*(H)$ is
finitely generated.

The symplectic chain complex $SC^*(H)$ is the Novikov completion
of $CF^*(H)$:
$$
\begin{array}{ll}
SC^*=CF^*(\!(t)\!) & = \left\{ \sum_{j=0}^{\infty} c_{j} t^{a_j} :
c_{j} \in CF^*, \lim a_j = \infty \right\} \\ & = \left\{
\sum_0^{N} \lambda_i y_{i}: \lambda_i \in \Lambda, N\in \N, y_{i}
\, \textrm{is a 1--periodic orbit of } X_H \right\}.
\end{array}
$$

The differential $\delta$ is defined by
$$
\delta \left( \sum_{i=0}^{N} \lambda_i y_{i}\right) =
\sum_{i=0}^{N} \sum_{u\in \mathcal{M}_0(x,y_{i})} \epsilon(u)\,
t^{E(u)}\lambda_i x
$$
where $\mathcal{M}_0(x,y_{i})$ is the $0-$dimensional component of
the Floer trajectories connecting $x$ to $y_i$, and $\epsilon(u)$
are signs depending on orientations. The sum is well-defined
because there are only finitely many generators $x$, and below any
energy bound $E(u)\leq K$ the moduli space $\mathcal{M}_0(x,y_i)$
is compact and therefore finite.

\begin{lemma}
$SC^*(H)$ is a chain complex, i.e. $\delta\circ\delta = 0$. We
denote the cohomology of $(SC^*(H),\delta)$ by $SH^*(H)$.
\end{lemma}
\begin{proof}
This involves a standard argument (see Salamon \cite{Salamon}).
Observe Figure \ref{Figure Compactness}. The $1-$dimensional
moduli space $\mathcal{M}_1(x,y)$ has a compactification, such
that the boundary consists of pairs of Floer trajectories joined
at one end. Observe that $E(\cdot)$ is additive with respect to
concatenation and $E(u)$ is invariant under homotoping $u$
relative ends. Therefore, in Figure \ref{Figure Compactness}, $
E(u_1') + E(u_2') = E(u_1'') + E(u_2'')$. Since
$\epsilon(u_1')\epsilon(u_2')=-\epsilon(u_1'')\epsilon(u_2'')$, we
deduce
$$
\epsilon(u_1') \, t^{E(u_1')}  \, \epsilon(u_2') \, t^{E(u_2')} =
- \epsilon(u_1'') \, t^{E(u_1'')} \, \epsilon(u_2'') \,
t^{E(u_2'')}.
$$
Thus the broken trajectories contribute opposite
$\Lambda-$multiples of $x$ to $\delta(\delta y)$ for each
connected component of $\mathcal{M}_1(x,y)$. Hence, summing over
$x,y'$,
$$
\delta(\delta y)=\sum_{(u_1',u_2')\in \mathcal{M}_0(x,y') \times
\mathcal{M}_0(y',y)} \epsilon(u_1') \, t^{E(u_1')}  \,
\epsilon(u_2') \, t^{E(u_2')} \, x=0.\qedhere
$$
\end{proof}
%
%
%
\subsection{Continuation Maps}\label{Subsection Continuation Maps}
%
Under suitable conditions on two Hamiltonians $H_{\pm}$, it is
possible to define a \emph{continuation homomorphism}
$$\varphi: SC^*(H_{+}) \to SC^*(H_{-}).$$
This involves counting \emph{parametrized Floer trajectories}, the
solutions of
$$
\partial_s v + J_s(\partial_t v - X_{H_s}) = 0.
$$
Here $J_s$ are $\omega-$compatible almost complex structures of
contact type and $H_s$ is a homotopy from $H_{-}$ to $H_{+}$, such
that $(H_s,J_s)=(H_{-},J_{-})$ for $s \ll 0$ and
$(H_s,J_s)=(H_{+},J_{+})$ for $s \gg 0$. The conditions on $H_s$
will be described in Theorem \ref{Theorem Chain map}.

If $x$ and $y$ are respectively $1$-periodic orbits of $X_{H_{-}}$
and $X_{H_{+}}$, then let $\mathcal{M}(x,y)$ be the moduli space
of such solutions $v$ which converge to $x$ and $y$ at the ends.
This time there is no freedom to reparametrize $v$ in the
$s-$variable.

The continuation map $\varphi$ on a generator $y \in
\textnormal{Zeros}(dA_{H_{+}})$ is defined by
$$
\varphi(y) = \sum_{v\in \mathcal{M}_0(x,y)} \epsilon(v)\,
t^{E_0(v)}\, x
$$
where $\mathcal{M}_0(x,y)$ is the $0-$dimensional part of
$\mathcal{M}(x,y)$, $\epsilon(v)\in \{ \pm 1\}$ are orientation
signs and the power of $t$ in the above sum is
$$
\begin{array}{ll}
E_0(v) & =- \int_{-\infty}^{\infty} \partial_s A_{H_s} (v) \, ds
\\
& = \int |\partial_s v|_{g_s}^2 \, ds\wedge dt - \int (\partial_s
H_s)(v) \, ds\wedge dt\\
& = \int v^*(\omega-dK\wedge dt),
\end{array}$$
where $K(s,m)=H_s(m)$. The last expression shows that $E_0(v)$ is
invariant under homotoping $v$ relative ends.
%
%
%
\subsection{Energy of parametrized Floer trajectories}
\label{Subsection Action and energy of a parametrized Floer
trajectory}
%
Let $H_s$ be a homotopy of Hamiltonians. For an $H_s-$Floer
trajectory the above weight $E_0(v)$ will be positive if $H_s$ is
monotone decreasing, $\partial_sH_s \leq 0$. The energy is
$$
E(v) = E_0(v) + \int (\partial_s H_s)(v) \, ds \wedge dt.
$$

If $\partial_s H_s\leq 0$ outside of a compact subset of
$\widehat{M}$, then a bound on $E_0(v)$ imposes a bound on $E(v)$.
Note that $E(v)$ is not invariant under homotoping $v$ relative
ends.
%
%
%
\subsection{Properties of continuation maps}
\label{Subsection Properties of continuation maps}
%
\begin{theorem}[Monotone homotopies]\label{Theorem Chain map}

Let $H_s$ be a homotopy between $H_{\pm}$ such that
\begin{enumerate}

\item on the collar $H_s = h_s(R)$ for large $R$;

\item $\partial_s h'_s \leq 0$ for $R\geq R_{\infty}$, some
$R_{\infty}$;

\item $h_s$ is linear for $R \geq R_{\infty}$ (the slope may be a
Reeb period, but not for $h_{\pm}$).
\end{enumerate}

Then, after a generic $C^2$-small time-dependent perturbation of
$(H_s,J_s)$,

\begin{enumerate}
\item all parametrized Floer trajectories lie in the compact
subset $$C=M\cup \{ R\leq R_{\infty} \} \subset \widehat{M};$$

\item $\mathcal{M}(x;y)$ is a smooth manifold;

\item $\mathcal{M}(x,y;K)=\{ v\in \mathcal{M}(x,y): E_0(v)\leq K
\}$ has a smooth compactification by broken trajectories, for any
constant $K\in \R$;

\item the continuation map $\varphi: SC^*(H_{+}) \to SC^*(H_{-})$
is well-defined;

\item $\varphi$ is a chain map.
\end{enumerate}
\end{theorem}

\begin{proof}
(1) is a consequence of the maximum principle, Lemma \ref{Lemma
Maximum principle}, and (2) is a standard transversality result.
Let
$$B_C=\max_{x\in C} \{\partial_s
H_s(x),0\}.$$

Suppose $H_s$ varies in $s$ precisely for $s\in [s_0,s_1]$. Since
all $v\in \mathcal{M}(x,y;K)$ lie in $C$, $\int
\partial_s H_s(v)\, ds\wedge dt \leq (s_1-s_0)B_C$, so there is
an a priori energy bound
$$
E(v) \leq K + (s_1-s_0)B_C.
$$
From this the compactness of $\mathcal{M}(x,y;K)$ follows by
standard methods.

The continuation map $\varphi$ involves a factor of $t^{E_0(v)}$.
The lower bound $E_0(v) \geq E(v) - (s_1-s_0)B_C$ guarantees that
as the energy $E(v)$ increases also the powers $t^{E_0(v)}$
increase, which proves (4).

Showing that $\varphi$ is a chain map is a standard argument which
involves investigating the boundaries of broken trajectories of
the $1-$dimensional moduli spaces $\mathcal{M}_1(x,y;K)$. A
sequence $v_n$ in some $1-$dimensional component of
$\mathcal{M}_1(x,y;K)$ will converge (after reparametrization) to
a concatenation of two trajectories $u^+ \# v$ or $v \# u^-$,
where $u^+\in \mathcal{M}_0^{H_+}(x,x')$, $v\in
\mathcal{M}_0(x',y)$, or respectively $v\in \mathcal{M}_0(x,y')$,
$u^-\in \mathcal{M}_0^{H_-}(y',y)$. Such solutions get counted
with the same weight
$$E_0(v_n) = E_0(u^+ \# v) = E_0(v \# u^-)$$
because $E_0$ is invariant under homotopies which fix the ends,
and $v_n, u^+ \# v, v \# u^-$ are homotopic since they belong to
the compactification of the same $1-$dimensional component of
$\mathcal{M}_1(x,y)$. Therefore, $\partial_{H_-} \circ \varphi =
\varphi \circ
\partial_{H_+}$ as required.\end{proof}
%
%
\subsection{Chain homotopies}
\label{Subsection Chain homotopies}
%

\begin{theorem}\label{Theorem Chain Homotopy}\strut
\begin{enumerate}
\item Given monotone homotopies $H_s$, $K_s$ from $H_-$ to $H_+$,
there is a chain homotopy $Y: SC^*(H_+) \to SC^*(H_-)$ between the
respective continuation maps:
$
\varphi - \psi = \partial_{H_-} Y + Y \partial_{H_+};
$

\item the chain map $\varphi$ defines a map on cohomology,
$$
[\varphi]: SH^*(H_+) \to SH^*(H_-),
$$
which is independent of the choice of the homotopy $H_s$;

\item the composite of the maps induced by homotoping $H_{-}$ to
$K$ and $K$ to $H_{+}$,
$$SC^*(H_{+}) \to SC^*(K) \to SC^*(H_{-}),
$$
is chain homotopic to $\varphi$ and equals $[\varphi]$ on
$SH^*(H_+)$;

\item the constant homotopy $H_s=H$ induces the identity on
$SC^*(H)$;

\item if $H_{\pm}$ have the same slope at infinity, then
$[\varphi]$ is an isomorphism.
\end{enumerate}
\end{theorem}

\begin{proof}
Let $(H_{s,\lambda})_{0\leq \lambda \leq 1}$ be a linear
interpolation of $H_s$ and $K_s$, so that $H_{s,\lambda}$ is a
monotone Hamiltonian for each $\lambda$. Consider the moduli
spaces $\mathcal{M}(x,y,\lambda)$ of parametrized Floer solutions
for $H_{s,\lambda}$. Let $Y$ be the oriented count of the pairs
$(\lambda,v)$, counted with weight $t^{E_0(v)}$, where
$0<\lambda<1$ and $v$ is in a component of
$\mathcal{M}(x,y,\lambda)$ of virtual dimension $-1$ (generically
$\mathcal{M}_{-1}(x,y,\lambda)$ is empty, but in the family
$\cup_{\lambda}\mathcal{M}_{-1}(x,y,\lambda)$ such isolated
solutions $(\lambda,v)$ can arise).

Consider a sequence $(\lambda_n,v_n)$ inside a $1-$dimensional
component of $\cup_{\lambda} \mathcal{M}(x,y,\lambda)$, such that
$\lambda_n\to \lambda$. If $\lambda= 0$ or $1$, then the limit of
the $v_n$ can break by giving rise to an $H_s$ or $K_s$ Floer
trajectory, and this breaking is counted by $\varphi - \psi$. If
$0<\lambda<1$, then the $v_n$ can break by giving rise to $u^-\#
v$ or $v\# u^+$, where $u^{\pm}$ are $H_{\pm}-$Floer trajectories
and the $v$ are as in the definition of $Y$. This type of breaking
is therefore counted by $\partial_{H_-} Y + Y
\partial_{H^+}$.

Both sides of the relation $\varphi - \psi =
\partial_{H_-} Y + Y \partial_{H_+}$ will count a (broken)
trajectory with the same weight because $E_0(\cdot)$ is a homotopy
invariant relative ends and the broken trajectories are all
homotopic, since they arise as the boundary of the same
$1-$dimensional component of
$\cup_{\lambda}\mathcal{M}(x,y,\lambda)$.

Claims (2) and (3) are standard consequences of (1) (see Salamon
\cite{Salamon}). Claim (4) is a consequence of the fact that any
non-constant Floer trajectory for $H_s=H$ would come in a
$1-$dimensional family of solutions, due to the translational
freedom in $s$. Claim (5) follows from (3) and (4): we can choose
$H_s$ to have constant slope for large $R$, therefore $H_{-s}$ is
also a monotone homotopy, and the composite of the chain maps
induced by $H_{s}$ and $H_{-s}$ is chain homotopic to the
identity.
\end{proof}
%
%
%
%
\subsection{Hamiltonians linear at
infinity}\label{Subsection Hamiltonians Linear At Infty}
%
Consider Hamiltonians $H^m$ which equal
$$
h^m(R) = mR+C
$$
for $R \gg 0$, where the slope $m>0$ is not the period of any Reeb
orbit. Up to isomorphism, $SH^*(H)$ is independent of the choice
of $C$ by Theorem \ref{Theorem Chain Homotopy}.

For $m_+<m_-$, a monotone homotopy $H_s$ defines a continuation
map
$$
\phi^{m_+m_-}: SC^*(H^{m_+}) \to SC^*(H^{m_-}),
$$
for example the homotopy $h_s(R) = m_s R + C_s$ for $R \gg 0$,
with $\partial_s m_s \leq 0$.

By Theorem \ref{Theorem Chain Homotopy} the continuation map $
[\phi^{m_+m_-}]: SH^*(H^{m_+}) \to SH^*(H^{m_-})$ on cohomology
does not depend on the choice of homotopy $h_s$. Moreover, such
continuation maps compose well: $\phi^{m_2m_3} \circ
\phi^{m_1m_2}$ is chain homotopic to $\phi^{m_1m_3}$ where
$m_1<m_2<m_3$, and so
$[\phi^{m_2m_3}] \circ [\phi^{m_1m_2}]= [\phi^{m_1m_3}].$
%
\subsection{Symplectic cohomology}
\label{Subsection Symplectic cohomology}
%
%
\begin{definition}
The symplectic cohomology is defined to be the direct limit
$$
SH^*(M,\omega) = \varinjlim SH^*(H)
$$
taken over the continuation maps between Hamiltonians linear at
infinity.
\end{definition}

Observe that $SH^*(M,\omega)$ can be calculated as the direct
limit
$$\lim_{k\to \infty} SH^*(H_k)$$
over the continuation maps $SH^*(H_k) \to SH^*(H_{k+1})$, where
the slopes at infinity of the Hamiltonians $H_k$ increase to
infinity as $k\to \infty$.
%
%
\subsection{The maps $\mathbf{c_*}$ from ordinary cohomology}
\label{Subsection The maps c from ordinary cohomology}
%
%

The symplectic cohomology comes with a map from the ordinary
cohomology of $M$ with coefficients in $\Lambda$,
$$c_*:
H^*(M;\Lambda) \to SH^*(M,\omega).
$$

We sketch the construction here, and refer to \cite{Ritter} for a
detailed construction. Fix a $\delta>0$ which is smaller than all
periods of the nonconstant Reeb orbits on $\partial M$. Consider
Hamiltonians $H^{\delta}$ which are $C^2$-close to a constant on
$M$ and such that on the collar $H^{\delta}=h(R)$ with constant
slope $h'(R)=\delta$.

A standard result in Floer cohomology is that, after a generic
$C^2$-small time-independent perturbation of $(H^{\delta},J)$, the
$1$-periodic orbits of $X_{H^{\delta}}$ and the connecting Floer
trajectories are both independent of $t\in S^1$. By the choice of
$\delta$ there are no $1$-periodic orbits on the collar, and by
the maximum principle no Floer trajectory leaves $M$. The Floer
complex $CF^*(H^{\delta})$ is therefore canonically identified
with the Morse complex $CM^*(H^{\delta})$, which is generated by
$\textnormal{Crit}(H^{\delta})$ and whose differential counts the
negative gradient trajectories of $H^{\delta}$ with weights
$t^{H^{\delta}(x_-)-H^{\delta}(x_+)}$. After the change of basis
$x \mapsto t^{H^{\delta}(x)}x$, the differential reduces to the
ordinary Morse complex defined over the ring $\Lambda$ which is
isomorphic to the singular cochain complex of $M$ with
coefficients in $\Lambda$. Thus
$$
SH^*(H^{\delta}) \cong HM^*(H^{\delta};\Lambda) \cong
H^*(M;\Lambda).
$$
Since $SH^*(H^{\delta})$ is part of the direct limit construction
of $SH^*(M,\omega)$, this defines a map $c_*:H^*(M;\Lambda) \to
SH^*(M,\omega)$ independently of the choice of $H^{\delta}$.
%
%
%
\subsection{Invariance under symplectomorphisms of contact type} \label{Section
Invariance under contactomorphisms}
%

\begin{definition}
Let $M,N$ be symplectic manifolds with contact type boundary. A
symplectomorphism $\varphi: \widehat{M} \to \widehat{N}$ is of
\emph{contact type at infinity} if on the collar
$$\varphi^*\theta_N = \theta_M + d(\textrm{compactly
supported function}).$$
This implies that at infinity $\varphi$ has the form
$$
\varphi(e^r,y) = (e^{r-f(y)}, \psi(y)),
$$
with $f: \partial M \to \R$ smooth, $\psi:
\partial M \to
\partial N$ a contactomorphism with $\psi^*\alpha_N = e^f \alpha_M$.
\end{definition}

Under such a map $\varphi: \widehat{M} \to \widehat{N}$, the Floer
solutions on $\widehat{N}$ for $(H,\omega_N,J_N)$ correspond
precisely to the Floer solutions on $\widehat{M}$ for
$(\varphi^*H,\omega_M,\varphi^*J_N)$ . However, for a Hamiltonian
$H$ on $\widehat{N}$ which is linear at infinity, the Hamiltonian
$\varphi^*H(e^r,y) = h(e^{r-f(y)})$ is not linear at infinity.
Thus we want to show that for this new class of Hamiltonians on
$\widehat{M}$ we still obtain the usual symplectic cohomology.

In order to relate the two symplectic cohomologies, we need a
maximum principle for homotopies of Hamiltonians which equal $H_s
= h_s(R_s)$ on the collar, where
$$R_s(e^r,y) = e^{r-f_s(y)},$$
and $f_s = f_-$, $h_s=h_-$ for $s\ll 0$ and $f_s = f_+, h_s=h_+$
for $s\gg 0$. We prove that if $h_-'\gg h_+'$ then one can choose
$h_s$ so that the maximum principle applies. We denote by $X_s$
the Hamiltonian vector field for $h_s$ and we assume that the
almost complex structures $J_s$ satisfy the contact type condition
$J_s^*\theta = dR_s$ for $e^r \gg 0$.

\begin{lemma}[Maximum principle]\label{Lemma Max principle for
contacto invariance} There is a constant $K>0$ depending only on
$f_s$ such that if $h_-' \geq K h_+'$ then it is possible to
choose a homotopy $h_s$ from $h_-$ to $h_+$ in such a way that the
maximum principle applies to the function
$$\rho(s,t) =
R_s(u(s,t)) = e^{r(u)-f_s(y(u))}$$
where $u$ is any local solution of Floer's equation $\partial_s u
+ J_s (\partial_t u - X_s)=0$ which lands in the collar $e^r \gg
0$, and where $J_s^*\theta = dR_s$ for $e^r \gg 0$. In particular,
a continuation map $SH^*(h_+) \to SH^*(h_-)$ can then be defined.
\end{lemma}
\begin{proof}
We will seek an equation satisfied by $\Delta \rho$. Using
$J_s^*dR_s = -\theta$ we obtain
$$\begin{array}{ll}
\partial_s \rho & = \partial_s R_s(u) + dR_s \cdot \partial_s u =
- \rho \partial_s f_s + dR_s \cdot \partial_s u\\
& = - \rho \partial_s f_s + dR_s \cdot (J_s(X_s - \partial_t u))
= - \rho \partial_s f_s - \theta(X_s) + \theta(\partial_t u), \\
\partial_t \rho & = dR_s \cdot \partial_t u = dR_s \cdot (X_s +
J_s \partial_s u) = J_s^*dR_s \cdot \partial_s u = - \theta
(\partial_s u).\end{array}
$$
Since $X_s = h_s'(\rho)\mathcal{R}$ and $\theta(\mathcal{R}(u)) =
\rho$, we deduce $\theta(X_s)=\rho h_s'(\rho)$ so
$$d^c \rho = d\rho
\circ i = -\partial_s \rho \, dt +
\partial_t \rho \, ds =  -u^*\theta + \rho h_s'(\rho) \, dt +\rho \partial_s f_s \, dt.$$
Therefore $dd^c \rho = - \Delta \rho \, ds\wedge dt = -u^*\omega +
F\, ds\wedge dt$ where
$$
F = h_s' \partial_s \rho + \rho \partial_s h_s' + \rho h_s''
\partial_s \rho + \partial_s \rho \partial_s f_s + \rho
\partial_s^2 f_s + \rho
d(\partial_s f_s)\cdot \partial_s u.
$$
We now try to relate $u^*\omega$ with $|\partial_s u|^2$:
$$
\begin{array}{ll}
|\partial_s u| ^2 & = \omega (\partial_s u, \partial_t u - X_s) =
u^*\omega - dH_s\cdot \partial_s u = u^*\omega - h_s' dR_s \cdot
\partial_s u \\ & = u^*\omega - h_s' \partial_s \rho - h_s' \rho \partial_s f_s
\end{array}
$$
where we used that $\partial_s \rho = - \rho
\partial_s f_s + dR_s \cdot \partial_s u$.

Thus, $\Delta \rho = u^* \omega - F$ equals
$$|\partial_s u|^2 +
h_s' \rho
\partial_s f_s - \rho \partial_s h_s' - \rho h_s''
\partial_s \rho - \partial_s \rho \partial_s f_s - \rho
\partial_s^2 f_s - \rho d(\partial_s f_s) \cdot \partial_s u.
$$
We may assume that $f$ is $C^2-$bounded by a constant $C>0$. Then
in particular
$$
|d(\partial_s f_s) \cdot \partial_s u| \leq \|d(\partial_s
f_s)\|_{\rho \times \partial M} \cdot |\partial_s u| \leq
\rho^{-1} \|d(\partial_s f_s)\|_{1\times \partial M} \cdot
|\partial_s u| \leq \rho^{-1} C |\partial_s u|.
$$

We deduce an inequality for $\Delta \rho$,
$$
\begin{array}{lll}
\Delta \rho + \textrm{ first order terms} & \geq & |\partial_s
u|^2-\rho \partial_s h_s'
- \rho (h_s' C + C) - C |\partial_s u| \\
& \geq & (|\partial_s u| - \frac{1}{2}C)^2 - \rho (\partial_s h_s'
+ h_s' C + C) - \frac{1}{4}C^2.
\end{array}
$$
Adding $-\rho h_s''\partial_s \rho$ to both sides and dropping the
squared bracket, we deduce
$$
\Delta \rho + \textrm{ first order terms} \geq - \rho
e^{-Cs}\,[\,\partial_s(e^{Cs}h_s')+e^{Cs}(C+C^2)\,].
$$
Thus a maximum principle will apply for $\rho$ if the right hand
side is non-negative. Now $f_s$ only depends on $s$ on a finite
interval $I$ of $s$-values. On the complement of $I$, $\partial_s
f_s\equiv 0$ so actually $\Delta \rho + \textrm{first order} \geq
-\rho
\partial_s h_s'$, so we only require $\partial_s
h_s'\leq 0$. On $I$ we need the last square-bracket to be
negative. Integrate this condition over $I$:
$$
h_+' e^{C\cdot\textrm{length}(I)} -h_-' \leq C'
$$
where $C'$ is a constant depending only on $C$ and the length of
$I$. Thus it is possible to choose $h_s$ satisfying the above
conditions, provided that $h_-'\gg h_+'$.
\end{proof}

\begin{theorem}\label{Theorem Invariance under Contactomorphs}
If $\varphi: \widehat{M} \to \widehat{N}$ is a symplectomorphism
of contact type at infinity, then $SH^*(M)\cong SH^*(N)$.
\end{theorem}
\begin{proof}
By identifying the Floer solutions via $\varphi$, the claim
reduces to showing that the symplectic cohomology $SH^*(M) =
\varinjlim SH^*(h)$ is isomorphic to the symplectic cohomology
$SH_f^*(M) = \varinjlim SH_f^*(h)$ which is calculated for
Hamiltonians of the form $H(e^r,y) = h(e^{r-f(y)})$, where the $h$
are linear at infinity and $f:\partial M \to \R$ is a fixed smooth
function.

Pick an interpolation $f_s$ from $f$ to $0$, constant in $s$ for
large $|s|$. We can inductively construct Hamiltonians $h_n$ and
$k_n$ on $\widehat{M}$ with $h_n'\gg k_n'$ and $k_{n+1}' \gg
h_{n}'$, which by Lemma \ref{Lemma Max principle for contacto
invariance} yield continuation maps
$$
\phi_n:SH_f^*(k_{n}) \to SH^*(h_{n}), \; \psi_n:SH^*(h_{n}) \to
SH_f^*(k_{n+1}).
$$
We can arrange that the slope at infinity of the $h_n$, $k_n$ grow
to infinity as $n\to \infty$, so that $SH_f^*(M) = \varinjlim
SH_f^*(k_n)$ and $SH^*(M) = \varinjlim SH^*(h_n)$.

The composites $\psi_n \circ \phi_n$ and $\phi_{n+1}\circ \psi_n$
are equal to the ordinary continuation maps $SH_f^*(k_{n}) \to
SH_f^*(k_{n+1})$ and $SH^*(h_{n}) \to SH^*(h_{n+1})$.

Therefore the maps $\phi_n$ and $\psi_n$ form a compatible family
of maps and so define
$$
\phi:SH_f^*(M) \to SH^*(M), \; \psi:SH^*(M) \to SH_f^*(M).
$$
The composites $\psi\circ \phi$ and $\phi\circ \psi$ are induced
by the families $\psi_n \circ \phi_n$ and $\phi_{n+1} \circ
\psi_n$, which are the ordinary continuation maps defining the
direct limits $SH_f^*(M)$ and $SH^*(M)$. Hence $\phi\circ \psi$,
$\psi\circ \phi$ are identity maps, and so $\phi,\psi$ are
isomorphisms.
\end{proof}
%
%
%
%
%
\subsection{Independence from choice of cohomology representative}
\label{Subsection Independence from choice of cohomology
representative}
%
%
%
\begin{lemma}\label{Lemma Moser deformation}
Let $\eta$ be a one-form supported in the interior of $M$. Suppose
there is a homotopy $\omega_{\lambda}$ through symplectic forms
from $\omega$ to $\omega+d\eta$. By Moser's lemma this yields an
isomorphism $\varphi:(\widehat{M},\omega+d\eta) \to
(\widehat{M},\omega)$, and therefore a chain isomorphism
$$
\varphi:SC^*(H,\omega+d\eta) \to SC^*(\varphi^*H,\omega),
$$
which is the identity on orbits outside $M$ and sends the orbits
$x$ in $M$ to $\varphi^{-1}x$.
\end{lemma}
%
%
%
%
%
%
%
%
\section{Twisted symplectic cohomology}
\label{Section Twisted symplectic cohomology}
%
%
\subsection{Transgressions}\label{SubsectionTransgressions}
Let $ev: \mathcal{L} M \times S^1 \to \mathcal{L} M$ be the
evaluation map. Define
$$\xymatrix{
\tau=\pi\circ ev^*: H^2(M;\R)\ar[r]^-{ev^*} & H^2(\mathcal{L} M
\times S^1;\R)\ar[r]^-{\pi} & H^1(\mathcal{L} M;\R) },
$$
where $\pi$ is the projection to the K\"{u}nneth summand.
Explicitly, $\tau \beta$ evaluated on a smooth path $u$ in
$\mathcal{L}M$ is given by
$$
\tau\beta(u) = \int \beta(\partial_s u, \partial_t u) \, ds\wedge
dt.
$$

In particular, $\tau\beta$ vanishes on time-independent paths in
$\mathcal{L} M$. If $M$ is simply connected, then $\tau$ is an
isomorphism. After identifying
$$H^1(\mathcal{L} M;\R) \cong
\textrm{Hom}_{\R}(H_1(\mathcal{L} M;\R),\R) \cong
\textrm{Hom}(\pi_1(\mathcal{L} M),\R)$$ and $\pi_1(\mathcal{L} M)
= \pi_2(M) \rtimes \pi_1(M)$, the $\tau \beta$ correspond
precisely to homomorphisms $\pi_2(M) \to \R$. In particular, if
$\beta$ is an integral class then this homomorphism is $f_*:
\pi_2(M) \to \Z$ where $f:M \to \CP$ is a classifying map for
$\beta$.
%
%
%
\subsection{Novikov bundles of coefficients}
\label{SubsectionNovikovbundles}
We suggest \cite{Whitehead} as a reference on local systems. Let
$\alpha$ be a singular smooth real cocycle representing $a \in
H^1(\mathcal{L} M; \R)$. The Novikov bundle
$\underline{\Lambda}_{\alpha}$ is the local system of coefficients
on $\mathcal{L} M$ defined by a copy $\Lambda_{\gamma}$ of
$\Lambda$ over each loop $\gamma \in \mathcal{L} M$ and by the
multiplication isomorphism
$$t^{-\alpha[u]} \co  \Lambda_{\gamma}
\to \Lambda_{\gamma'}$$
for each path $u$ in $\mathcal{L} M$ connecting $\gamma$ to
$\gamma'$. Here $\alpha[\cdot] \co C_1(\mathcal{L} M;\R) \to \R$
denotes evaluation on smooth singular one-chains, which is given
explicitly by
$$
\alpha[u] = \int \alpha(\partial_s u) \, ds.
$$
Changing $\alpha$ to $\alpha+df$ yields a change of basis
isomorphism $x \mapsto t^{f(x)}x$ for the local systems, so by
abuse of notation we write $\underline{\Lambda}_{a}$ and $a[u]$
instead of $\underline{\Lambda}_{\alpha}$ and $\alpha[u]$.

\begin{remark}
In \cite{Ritter} we used the opposite sign convention
$t^{\alpha[u]}$. In this paper we changed it for the following
reason. For a Liouville domain $(M,d\theta)$, the local system for
the action $1-$form $\alpha=dA_H$ acts on Floer solutions $u\in
\mathcal{M}(x,y;d\theta)$ by
$$
t^{-dA_H[u]} = t^{A_H(x) - A_H(y)}.
$$
Therefore large energy Floer solutions will occur with high powers
of $t$.
\end{remark}

We will be considering the (co)homology of $M$ or $\mathcal{L}M$
with local coefficients in the Novikov bundles, and we now mention
two recurrent examples. First consider a transgressed form $\alpha
= \tau\beta$ (see \ref{SubsectionTransgressions}). Since
$\tau(\beta)$ vanishes on time-independent paths,
$\underline{\Lambda}_{\tau\beta}$ pulls back to a trivial bundle
via the inclusion of constant loops $c \co M \to \mathcal{L}M$. So
for the bundle $c^*\underline{\Lambda}_{\tau\beta}$ we just get
ordinary cohomology with underlying ring $\Lambda$,
$$
H^*(M;c^*\underline{\Lambda}_{\tau(\beta)})\cong H^*(M;\Lambda).
$$

Secondly, consider a map $j\co L \to M$. This induces a map
$\mathcal{L}j \co \mathcal{L}L \to \mathcal{L}M$ which by the
naturality of $\tau$ satisfies
$ (\mathcal{L}j)^*\underline{\Lambda}_{\tau(\beta)} \cong
\underline{\Lambda}_{\tau(j^*\beta)}. $
For example if $\tau(j^*\beta)=0 \in H^1(\mathcal{L} L;\R)$ then
this is a trivial bundle, so the corresponding Novikov homology is
$$
H_*(\mathcal{L}
L;(\mathcal{L}j)^*\underline{\Lambda}_{\tau(\beta)}) \cong
H_*(\mathcal{L} L)\otimes\Lambda.
$$
%
%
\subsection{Twisted Floer
cohomology}\label{SubsectionNovikovFloerCohomology}
Let $(M^{2n},\theta)$ be a Liouville domain. Let $\alpha$ be a
singular cocyle representing a class in $H^1(\mathcal{L} M;\R)
\cong H^1(\mathcal{L} \widehat{M};\R)$. The Floer chain complex
for $H \in C^{\infty}(\widehat{M},\R)$ with twisted coefficients
in $\underline{\Lambda}_{\alpha}$ is the $\Lambda-$module freely
generated by the $1$-periodic orbits of $X_H$,
$$
CF^*(H;\underline{\Lambda}_{\alpha}) =\bigoplus \left\{ \Lambda x
: x \in \mathcal{L} \widehat{M},\; \dot{x}(t) = X_H(x(t))
\right\},
$$
and the differential $\delta$ on a generator $y \in
\textnormal{Crit}(A_H)$ is defined as
$$
\delta y = \sum_{u\in \mathcal{M}_0(x,y)} \epsilon(u)\,
t^{-\alpha[u]}\,  x,
$$
where $\epsilon(u)\in \{\pm 1\}$ are orientation signs and
$\mathcal{M}_0(x,y)$ is the $0-$dimensional component of Floer
trajectories connecting $x$ to $y$. It is always understood that
we perturb $(H,J)$ as explained in \ref{Subsection Transversality
and Compactness}.

The ordinary Floer complex (with underlying ring $\Lambda$) has no
weights $t^{-\alpha[u]}$ in $\delta$. These appear in the twisted
case because they are the multiplication isomorphisms $\Lambda_{x}
\to \Lambda_{y}$ of the local system
$\underline{\Lambda}_{\alpha}$ which identify the $\Lambda-$fibres
over $x$ and $y$ (see \ref{SubsectionNovikovbundles}).

\begin{proposition-definition}\cite{Ritter}
$CF^*(H;\underline{\Lambda}_{\alpha})$ is a chain complex:
$\delta\circ\delta = 0$, and its cohomology
$HF^*(H;\underline{\Lambda}_{\alpha})$ is a $\Lambda-$module
called twisted Floer cohomology.
\end{proposition-definition}
%
%
\subsection{Twisted symplectic
cohomology}\label{SubsectionNovikovSymplecticCohomology}
%
%
\begin{proposition}[Twisted continuation
maps, \cite{Ritter}] For the twisted Floer cohomology of
$(M,d\theta)$, Theorem \ref{Theorem Chain map} continues to hold
for the continuation maps
$ \phi: CF^*(H_{+};\underline{\Lambda}_{\alpha}) \to
CF^*(H_{-};\underline{\Lambda}_{\alpha}) $
defined on generators $y \in \textnormal{Crit}(A_{H_{+}})$ by
$$
\phi(y)= \sum_{v \in \mathcal{M}_0(x,y)} \epsilon(v)\,
t^{-\alpha[v]} \, x.
$$
\end{proposition}
\begin{definition}
The twisted symplectic cohomology of $(M,d\theta;\alpha)$ is
$$
SH^*(M,d\theta;\underline{\Lambda}_{\alpha}) = \varinjlim
HF^*(H,d\theta;\underline{\Lambda}_{\alpha}),
$$
where the direct limit is over the twisted continuation maps
between Hamiltonians $H$ which are linear at infinity.
\end{definition}
%
%
%
%
\subsection{Independence from choice of cohomology representative}
\label{Subsection Independence from choice of cohomology
representative}
%
%
%
\begin{lemma}\label{Lemma Moser deformation twisted}
Let $f_H\in C^{\infty}(\mathcal{L}\widehat{M},\R)$ be an
$H-$dependent function. Then the change of basis isomorphisms
$x\mapsto t^{f_H(x)}x$ of the local systems induce chain
isomorphisms
$$SC^*(H,d\theta;\underline{\Lambda}_{\alpha})
\cong  SC^*(H,d\theta;\underline{\Lambda}_{\alpha+df_H})
$$
which commute with the twisted continuation maps.
\end{lemma}
%
%
%
\subsection{Twisted maps $\mathbf{c_*}$ from ordinary cohomology}
\label{Subsection The twisted maps c from ordinary cohomology}
%
%

The twisted symplectic cohomology comes with a map from the
induced Novikov cohomology of $M$,
$$c_*:
H^*(M;c^*\underline{\Lambda}_{\alpha}) \to
SH^*(M;d\theta,\underline{\Lambda}_{\alpha}).$$
The construction is analogous to \ref{Subsection The maps c from
ordinary cohomology}, and was carried out in detail in
\cite{Ritter}. The map $c_*$ comes automatically with the direct
limit construction of
$SH^*(M;d\theta,\underline{\Lambda}_{\alpha})$, since for the
Hamiltonian $H^{\delta}$ described in \ref{Subsection The maps c
from ordinary cohomology} we have
$$
H\!F^*(H^{\delta};\underline{\Lambda}_{\alpha}) \cong
H\!M^*(H^{\delta};c^*\underline{\Lambda}_{\alpha}) \cong
H^*(M;c^*\underline{\Lambda}_{\alpha}).$$
%
%
%
%
%
\subsection{Twisted Functoriality}
\label{Subsection Twisted Viterbo Functoriality}
%
%
In \cite{Ritter} we proved the following variant of Viterbo
functoriality \cite{Viterbo1}, which holds for \emph{Liouville
subdomains} $(W^{2n},\theta') \subset (M^{2n},\theta)$. These are
Liouville domains for which $\theta - e^\rho\theta'$ is exact for
some $\rho \in \R$. The standard example is the Weinstein
embedding $DT^*L \hookrightarrow DT^*N$ of a small disc cotangent
bundle of an exact Lagrangian $L \hookrightarrow DT^*N$ (see
\cite{Ritter}).
\begin{theorem}\cite{Ritter}\label{TheoremViterboTwistedFunctoriality}
Let $ i:(W^{2n},\theta') \hookrightarrow (M^{2n},\theta) $ be a
Liouville embedded subdomain. Then there exists a map
$$SH^*(i):SH^*(W,d\theta';\underline{\Lambda}_{(\mathcal{L}i)^*\alpha})
\leftarrow SH^*(M,d\theta;\underline{\Lambda}_{\alpha})$$
which fits into the commutative diagram
$$
\xymatrix{
SH^*(W,d\theta';\underline{\Lambda}_{(\mathcal{L}i)^*\alpha})
\ar@{<-}[r]^-{SH^*(i)} \ar@{<-}[d]^-{c_*} &
SH^*(M,d\theta;\underline{\Lambda}_{\alpha}) \ar@{<-}[d]^-{c_*} \\
H^*(W;c^*\underline{\Lambda}_{(\mathcal{L}i)^*\alpha})
\ar@{<-}[r]^-{i^*} & H^*(M;c^*\underline{\Lambda}_{\alpha}) }
$$
\end{theorem}
The map $SH^*(i)$ is constructed using a ``step-shaped"
Hamiltonian, as in Figure \ref{FigureViterbo}, which grows near
$\partial W$ and reaches a slope $a$, then becomes constant up to
$\partial M$ where it grows again up to slope $b$. By a careful
construction, with $a\gg b$, one can arrange that all orbits in
$W$ have negative action with respect to $(d\theta,H)$, and for
orbits outside of $W$ they have positive actions. The map
$SH^*(i)$ is then the limit, as $a\gg b\to \infty$, of the action
restriction maps which quotient out by the generators of positive
action.
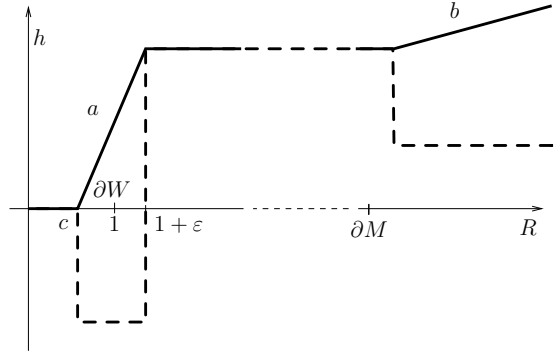
\begin{figure}[ht] \centering \scalebox{0.6}{
\input{viterbo1.pstex_t} }
\caption{The solid line is a diagonal-step shaped Hamiltonian with
$a\gg b$. The dashed line is the action $A(R)=-R h'(R) +
h(R)$.}\label{FigureViterbo}
\end{figure}
\begin{theorem}\label{Theorem Viterbo Functoriality}\cite{Ritter}
Let $(M,d\theta)$ be a Liouville domain and let $L \subset M$ be
an exact orientable Lagrangian submanifold. By Weinstein's
Theorem, this defines a Liouville embedding $j: (DT^*L,d\theta)\to
(M,d\theta)$ of a small disc cotangent bundle of $L$. Then for all
$\alpha \in H^1(\mathcal{L}M;\R)$ there exists a commutative
diagram
$$
\xymatrix{
H_{n-*}(\mathcal{L}L;\underline{\Lambda}_{(\mathcal{L}j)^*\alpha})
\cong
SH^{*}(T^*L,d\theta;\underline{\Lambda}_{(\mathcal{L}j)^*\alpha})
\ar@{<-}^{c_*}@<-10ex>[d] \ar@{<-}[r]^-{} &
SH^*(M,d\theta;\underline{\Lambda}_{\alpha}) \ar@{<-}^{c_*}@<1ex>[d] \\
H_{n-*}(L;c^*\underline{\Lambda}_{(\mathcal{L}j)^*\alpha}) \cong
H^*(L;c^*\underline{\Lambda}_{(\mathcal{L}j)^*\alpha})
\ar@{<-}[r]^-{} & H^*(M;c^*\underline{\Lambda}_{\alpha}) }
$$
where the left vertical map is induced by the inclusion of
constant loops $c:L \to \mathcal{L}L$. If $c^*\alpha=0$ then the
bottom map is the pullback $H^*(L) \otimes \Lambda \leftarrow
H^*(M) \otimes \Lambda$.
\end{theorem}
\begin{corollary}\label{Corollary Functoriality trick}
Let $(M,d\theta)$ be a Liouville domain and let $L\subset M$ be an
exact orientable Lagrangian. Suppose $\beta \in
H^2(\widehat{M};\R)$ is such that $\tau(j^*\beta) =0 \in
H^1(\mathcal{L}L;\R)$. Then there is a commutative diagram
$$
\xymatrix{ H_{n-*}(\mathcal{L}L) \otimes \Lambda
\ar@{<-_{)}}_{c_*}@<-3ex>[d] \ar@{<-}[r] &
SH^{*}(M,d\theta;\underline{\Lambda}_{\tau\beta}) \ar@{<-}^{c_*}@<1ex>[d] \\
H_{n-*}(L) \otimes \Lambda\cong H^*(L) \otimes \Lambda
\ar@{<-}[r]^-{j^*} & H^*(M) \otimes \Lambda }
$$
Therefore $SH^*(M,d\theta;\underline{\Lambda}_{\tau\beta})$ cannot
vanish since $c_* j^* 1 = c_*1 \neq 0$.
\end{corollary}
\begin{remark}\label{Remark Unorientable exact Lagrangians} \textbf{Unorientable exact Lagrangians.}
In Theorem \ref{Theorem Viterbo Functoriality} we assumed that the
Lagrangian is orientable. However, the result easily extends to
the unorientable case: instead of using $\Z$ coefficients we use
$\Z_2$ coefficients. This means that the moduli spaces do not need
to be oriented and we can drop all orientation signs in the
definitions of the differentials for the Floer complexes and the
Morse complexes. The Novikov ring is now defined by
$$\textstyle \Lambda = \{ \sum_{n=0}^{\infty} a_n t^{r_n}: a_n \in \Z_2, r_n
\in \R, r_n\to \infty \}.$$
Note that the Novikov one-form $\alpha$ is still chosen in
$H^1(\mathcal{L}M;\R)$.

This is particularly interesting in dimension four since
$H^2(L;\R)=0$ for unorientable $L^2 \subset M^4$, therefore the
transgression vanishes. In particular the pullback of any
transgression from $M$ will vanish on $L$. This immediately
contradicts Corollary \ref{Corollary Functoriality trick} if
$SH^*(M,d\theta;\underline{\Lambda}_{\tau \beta})=0$. For example
in \cite{Ritter} we proved that
$SH^*(T^*S^2,d\theta;\underline{\Lambda}_{\tau\beta})=0$ for any
non-zero $\beta \in H^2(S^2;\R)$. Therefore there can be no
unorientable exact Lagrangians in $T^*S^2$.
\end{remark}
%
%
%
%
\section{Grading of symplectic cohomology}
\label{Section Grading of symplectic cohomology}
%
\subsection{Maslov index and Conley-Zehnder grading}
\label{Subsection Maslov index and Conley-Zehnder index}
%
%
We assume that $c_1(M)=0$: this condition will ensure that the
symplectic cohomology has a $\Z-$grading defined by the
Conley-Zehnder index.

Since $c_1(M)=0$, we can choose a trivialization of the canonical
bundle $\mathcal{K}=\Lambda^{n,0}T^*M$. Then over any $1-$periodic
Hamiltonian orbit $\gamma$, trivialize $\gamma^*TM$ so that it
induces an isomorphic trivialization of $\mathcal{K}$. Let
$\phi_t$ denote the linearization $D\varphi^t(\gamma(0))$ of the
time $t$ Hamiltonian flow written in a trivializing frame for
$\gamma^*TM$.

Let $\textrm{sign}(t)$ denote the signature of the quadratic form
$$
\omega(\cdot,\partial_t \phi_t\cdot):
\textrm{ker}(\phi_t-\textrm{id}) \to \R,
$$
assuming we perturbed $\phi_t$ relative endpoints to make the
quadratic form non-degenerate and to make
$\textrm{ker}(\phi_t-\textrm{id})=0$ except at finitely many $t$.

The Maslov index $\mu(\gamma)$ of $\gamma$ is
$$
\mu(\gamma) = \frac{1}{2}\, \textrm{sign}(0) + \sum_{0<t<1}
\textrm{sign}(t) + \frac{1}{2}\, \textrm{sign}(1).
$$

The Maslov index is invariant under homotopy relative endpoints,
and it is additive with respect to concatenations. If $\phi_t$ is
a loop of unitary transformations, then its Maslov index is the
winding number of the determinant, $\det \phi_t: \mathcal{K} \to
\mathcal{K}$. For example $\phi_t = e^{2\pi i t} \in U(1)$ for
$t\in [0,1]$ has Maslov index $1$.

In our applications, $\gamma$ will often not be an isolated orbit.
It will typically lie in an $S^1-$worth or an $S^3-$worth of
orbits. In this case it is possible to make a small time-dependent
perturbation of $H$ so that $\gamma$ breaks up into two isolated
orbits whose Maslov indices get shifted by $\pm
\textrm{dim}(S^1)/2$ or $\pm \textrm{dim}(S^3)/2$ respectively.

The grading we use on $SH^*$ is the Conley-Zehnder index, defined
by
$$
|\gamma| = \frac{\textrm{dim}(M)}{2} - \mu(\gamma).
$$
This grading agrees with the Morse index when $H$ is a generic
$C^2-$small Hamiltonian and $\gamma$ is a critical point of $H$.
%
%
%
%
\section{Deformation of the Symplectic cohomology}
\label{Section Deformation of the Symplectic cohomology}
Let $\beta$ be a compactly supported two-form representing a class
in $H^2(M;\R)$ such that $d\theta + \beta$ is symplectic. We want
to construct an isomorphism between the non-exact symplectic
cohomology and the twisted symplectic cohomology
$$SH^*(H,d\theta+\beta) \cong
SH^*(H,d\theta;\underline{\Lambda}_{\tau\beta}).$$
We will show that this holds if $d\theta+s\beta$ is symplectic for
$0\leq s \leq 1$. For example, it will always hold if
$\|\beta\|<1$.
%
%
\subsection{Outline of the argument}
\label{Subsection Outline of the argument}
%
Let $H^m$ denote a Hamiltonian which only depends on $R$ on the
collar and which has slope $m$ at infinity. Choosing $H^m$ generic
and $C^2-$small inside $M$ ensures that the only $1-$periodic
Hamiltonian orbits inside $M$ are the critical points of $H^m$. We
will prove that we may assume that the critical points lie outside
the support of $\beta$. Therefore $SC^*(H^m,d\theta+\beta)$ and
$SC^*(H^m, d\theta; \underline{\Lambda}_{\tau\beta})$ have the
same generators: the critical points of $H^m$ and the $1-$periodic
Hamiltonian orbits lying in the collar (we used that $\supp \beta
\subset M$).

We will build chain isomorphisms
$$
\psi_{\mu}^m: SC^*(H^m,d\theta+\beta) \to SC^*(H^m, d\theta;
\underline{\Lambda}_{\tau\beta})
$$
which are defined for a sufficiently large parameter $\mu$; which
are independent of $\mu$ on homology, say $\psi^m=\psi^m_{\mu}$;
and which commute with the continuation maps
$$
\xymatrix{SH^*(H^m,d\theta+\beta) \ar@{->}[r]^{\psi^m}\ar@{->}[d]
& SH^*(H^m,d\theta;\underline{\Lambda}_{\tau\beta})
\ar@{->}[d] \\
SH^*(H^{m'},d\theta+\beta) \ar@{->}[r]^{\psi^{m'}} &
SH^*(H^{m'},d\theta;\underline{\Lambda}_{\tau\beta})}
$$

Therefore, by exactness of direct limits, $\psi = \varinjlim
\psi^m$ is the desired isomorphism
$$
\psi: SH^*(M,d\theta+\beta) \to SH^*(M, d\theta,
\underline{\Lambda}_{\tau\beta}).
$$

The parameter $\mu$ arises in the construction of the maps
$\psi_{\mu}^m$ because for large $\mu$ the identity map provides a
natural chain isomorphism
$$\textrm{id}: SC^*(H^m,d\theta+\mu^{-1}\beta) \cong
SC^*(H^m,d\theta;\underline{\Lambda}_{\mu^{-1}\tau\beta}).
$$
This is proved by showing that the moduli spaces $\mathcal{M}(x,y;
d\theta+\lambda \beta)$ form a $1-$parameter family joining
$\mathcal{M}(x,y;d\theta+\mu^{-1}\beta)$ to
$\mathcal{M}(x,y;d\theta)$.

To define the maps $\psi_{\mu}^m$ we therefore just need to deform
$d\theta+\beta$ to $d\theta+\mu^{-1}\beta$. On the twisted side,
there are no difficulties:
$$SH^*(H^m,d\theta;\underline{\Lambda}_{\tau\beta}) \cong
SH^*(H^m,d\theta;\underline{\Lambda}_{\mu^{-1}\tau\beta})$$
is just a rescaling $t\mapsto t^{(\mu^{-1})}$.

For the non-exact symplectic cohomology we first combine the
Liouville flow $\varphi_{\mu}$ for time $\log \mu$ and a rescaling
of the metric by $\mu^{-1}$. This will change $d\theta+\beta$ to
$d\theta+\mu^{-1}\varphi_{\mu}^*\beta$. Then we want to make a
Moser deformation from $d\theta+\mu^{-1}\varphi_{\mu}^*\beta$ to
$d\theta+\mu^{-1}\beta$, so we need a deformation through
symplectic forms without changing the cohomology class. This is
possible if $d\theta+s\beta$ is symplectic for $0\leq s\leq 1$.

\begin{lemma}\label{Lemma deforming the sympl form}
If $d\theta+s\beta$ is symplectic for $0\leq s \leq 1$, then it is
possible to deform $d\theta+\mu^{-1}\beta$ to
$d\theta+\mu^{-1}\varphi_{\mu}^*\beta$ through symplectic forms
within its cohomology class.
\end{lemma}
\begin{proof}
Since $d\theta+s\beta$ are symplectic for $0\leq s \leq 1$, so are
$$\omega_s=(s\mu)^{-1}\varphi_{s\mu}^*(d\theta+s\beta) =
d\theta+\mu^{-1}\varphi_{s\mu}^*\beta
$$
for $\frac{1}{\mu}\leq s \leq 1$. It remains to show that
$\partial_s \omega_s$ is exact. By Cartan's formula,
$$
\partial_s \omega_s = \varphi_{s\mu}^*\mathcal{L}_{Z/s\mu}
\beta = \varphi_{s\mu}^*(i_{Z/s\mu} d\beta + di_{Z/s\mu} \beta)=d
\varphi_{s\mu}^*(i_{Z/s\mu} \beta).\qedhere$$
\end{proof}

The argument hides a small technical challenge. The changes in
symplectic forms will change the Hamiltonian $H^m$ (without
affecting the slope at infinity). Since the $1-$parameter family
argument heavily depends on $H^m$, it is not clear that the same
large $\mu$ works for all Hamiltonians of a given slope. Therefore
we first apply a continuation isomorphism to change the
Hamiltonian back to the original $H^m$. Now it is no longer clear
that $\psi_{\mu}^m$ is independent of $\mu$ on homology, and when
we take the direct limit of continuation maps as $m\to \infty$ it
is not clear that the same choice of $\mu$ will work for different
Hamiltonians. Thus it is necessary to prove that the construction
is independent of $\mu$.

The $1-$parameter family of moduli spaces argument is presented in
\ref{Subsection 1-parameter family}. We will need several
preliminary results: the Palais-Smale Lemma (\ref{Subsection
Palais-Smale Lemma}); the Lyapunov property for the action
functional (\ref{Subsection Lyapunov property of the action
functional}); an a priori energy estimate (\ref{Subsection A
priori energy estimate}) and a transversality result
(\ref{Subsection Transversality for deformations}). In section
\ref{Subsection Construction of the isomorphism} we will construct
the maps $\psi_{\mu}^m$.
%
%
\subsection{Metric rescaling}\label{Subsection Metric rescaling}
%
\begin{lemma}\label{Lemma Metric rescaling} Let $\mu >0$. There is a natural
identification
$$SC^*(H,\omega) \to SC^*(\mu H,\mu\, \omega),$$
induced by the change of ring isomorphism $\Lambda \to \Lambda, t
\mapsto t^{\mu}$.
\end{lemma}
\begin{proof}
Under the rescaling, $X_H$ does not change, so the Floer equations
don't change. The energy functional gets rescaled by $\mu$, so a
Floer trajectory contributes a factor $t^{\mu E(u)} =
(t^{\mu})^{E(u)}$ to the differential instead of $t^{E(u)}$.
\end{proof}
%
%
%
%
\subsection{Palais-Smale Lemma}\label{Subsection Palais-Smale Lemma}
%
Let $X_t$ be a time-dependent vector field. Define
$$
F:\mathcal{L}M \to \bigcup_{x\in \mathcal{L}M} x^*TM, \;
F(x)(t)=\dot{x}(t)-X_t(x(t)).
$$
The solutions of $F(x)=0$ are precisely the $1-$periodic orbits of
$X_t$. The following standard result (see Salamon \cite{Salamon})
ensures that $F$ is small only near such solutions.
\begin{lemma}Let $M$ be a compact Riemannian manifold, and $X_t$ a time-dependent vector
field on $M$ whose $1-$periodic orbits form a discrete set. Then
\begin{enumerate}
\item A sequence $x_n \in \mathcal{L}M$ with $\| F(x_n) \|_{L^2}
\to 0$ has a subsequence converging in $C^0$ to a solution of
$F(x)=0$.

\item For any $\epsilon>0$ there is a $\delta>0$ such that $\|
F(y) \|_{L^2}<\delta$ implies that there is some solution of
$F(x)=0$ close to $y$,
$
\sup_{t \in S^1} \textrm{dist}\, (x(t),y(t)) < \epsilon.
$
\end{enumerate}
\end{lemma}
\begin{corollary}\label{Corollary Palais-Smale}
Let $(M,d\theta)$ be a Liouville domain, fix $J$ as in
\ref{Subsection Symplectic manifolds with contact type boundary}.
Let $H_t$ be a time-dependent Hamiltonian on $\widehat{M}$ such
that $H_t=h(R)$ is linear with generic slope for $R\gg 0$. Then
for any $\delta>0$ there is an $\epsilon>0$ such that any smooth
loop $x:S^1 \to \widehat{M}$ with $\|F(x)\|<\delta$ will be within
distance $\epsilon$ of some $1$-periodic orbit of $H_t$.
\end{corollary}
%
%
%
%
\subsection{Lyapunov property of the action functional}
\label{Subsection Lyapunov property of the action functional}
%
Let $(M,d\theta)$ be a Liouville domain, and pick $J$ as in
\ref{Subsection Symplectic manifolds with contact type boundary}.
The metric we use will be $d\theta(\cdot,J\cdot)$, and denote by
$|\cdot|$ the norm and by $\| \cdot \|$ the $L^2-$norm integrating
over time. Let $X$ be the Hamiltonian vector field for
$(H,d\theta)$, where $H$ is linear at infinity, and recall
$F(x)=\partial_t x-X(x)$.

Let $\beta$ be a closed two-form compactly supported in $M$ such
that $d\theta+\beta$ is symplectic. Denote $X_{\beta}$ the
Hamiltonian vector field for $(H,d\theta+\beta)$, and let
$$F_{\beta}(x)=\partial_t x - X_{\beta}(x).$$

Let $\| \beta \| = \sup |\beta (Y,Z)|$ taken over all vectors
$Y,Z$ of norm $1$. We will also use the notation $Y_{\supp \beta}$
for a vector field $Y$, where
$$Y_{\supp \beta}(m)=Y(m) \textrm{ if } m\in \supp \beta, \textrm{ and } Y_{\supp \beta}(m)=0 \textrm{ otherwise}.$$
\begin{lemma}\label{Lemma Difference of F}
Let $V$ be a neighbourhood containing the $1-$periodic orbits of
$X$ in $M$, and let $\beta$ be a closed $2$-form compactly
supported in $M$ and vanishing on $V$.
\begin{enumerate}
\item If $\| \beta \|<1$ then
$\displaystyle \| F_{\beta}(x) - F(x) \| \leq \frac{\| \beta
\|}{1-\| \beta \|} \,
 \| X_{\supp \beta} \|.
$
\item  There is a $\delta>0$ depending on $(M,H,d\theta,J,V)$, but
not on $\beta$, such that
$$
\|F(x)\|<\delta \Longrightarrow x \textrm{ lies in }V \textrm{ or
outside }M, \textrm{ so } F_\beta(x)=F(x).
$$
\item If $\|\beta\|$ is sufficiently small, then
$\|F_{\beta} -F\| \leq \frac{1}{3} \| F\|$ and $\|F_{\beta}\| \leq
2 \| F\|.$
\end{enumerate}
\end{lemma}
\begin{proof}
Observe that $F_{\beta}-F = X-X_{\beta}$ and that $
d\theta(X-X_{\beta},.) = \beta(X_{\beta},.) $, so
$$
| X-X_{\beta}|^2 =  \beta(X_{\beta},J(X-X_{\beta}))
 \leq  \| \beta \| \cdot | X_{\beta} | \cdot | X-X_{\beta} |.
$$
Dividing out by $|X-X_{\beta}|$ gives
$ | X-X_{\beta}| \leq \| \beta \| \cdot | X_{\beta}|$.
From $|X_{\beta}| \leq |X|+|X-X_{\beta}|$ we deduce that
$|X_{\beta}| \leq \frac{1}{1-\| \beta \|}\, |X|$. Therefore
$$
| F_{\beta}(x) - F(x) | \leq \frac{\| \beta \|}{1-\| \beta \|} \,
 | X_{\supp \beta}(x) |,
$$
since $F_{\beta}-F=X-X_{\beta}$ vanishes at $(x,t)$ if the loop
$x$ lies outside the support of $\beta$ at time $t$. The first
claim follows, and the second follows by Corollary \ref{Corollary
Palais-Smale}.

Let $C= \sup |X_{\supp \beta}|$. Then, whenever $\|F\| \geq
\delta$,
$$
\|F_{\beta} - F\| \leq \frac{\|\beta\|}{1-\|\beta\|} C \leq
\frac{\|\beta\|}{1-\|\beta\|} \frac{C}{\delta} \|F\| \leq
\frac{1}{3} \|F\|
$$
for small enough $\|\beta\|$. The last claim then follows from
$(2)$.
\end{proof}
For $\| \beta \|<1$, $d\theta+\beta$ is symplectic and
$(d\theta+\beta)(\cdot, J\cdot)$ is positive definite but may not
be symmetric. By symmetrizing, we obtain a metric
$$
\widetilde{g}_{\beta}(V,W) = \frac{1}{2} [(d\theta + \beta)(V,JW)+
(d\theta+\beta)(W, J V)].
$$
There is a unique endomorphism $B$ such that
$\widetilde{g}_{\beta}(BV,W)=(d\theta+\beta)(V,W)$, and this
yields an almost complex structure $J_{\beta}=(-B^2)^{-1/2}B$
compatible with $d\theta+\beta$, inducing the metric
$$
g_{\beta}(V,W)=(d\theta+\beta)(V,J_{\beta}W) =
\widetilde{g}_{\beta}((-B^2)^{1/2}V,W).
$$
For sufficiently small $\|\beta\|$, $J_{\beta}$ is $C^0$-close to
$J$ and is equal to $J$ outside the support of $\beta$, so in
particular $g_{\beta}$ induces a norm $|\cdot |_{\beta}$ which is
equivalent to the norm $|\cdot |$. Moreover, on the support of
$\beta$ we may perturb $J_{\beta}$ among
$(d\theta+\beta)-$compatible almost complex structures so that
transversality holds for $(d\theta+\beta)-$Floer trajectories. For
convenience, we use the abbreviations
$$
\delta J = J_{\beta} - J \qquad\quad \delta F = F_{\beta} - F.
$$
\begin{theorem}\label{Theorem Lyapunov property}
Let $V$ be a neighbourhood containing the $1-$periodic orbits of
$X$ in $M$, and let $\beta$ be a closed $2$-form compactly
supported in $M$ and vanishing on $V$. Then for sufficiently small
$\|\beta\|$,
$$
\partial_s A(u) \leq -\frac{1}{2} \| F(u) \|^2
$$
for all $u\in \mathcal{M}(x,y;d\theta+\beta,H)$, where $A(x)=-\int
x^*\theta+\int H(x)\, dt$ is the action functional for
$(H,d\theta)$. In particular, $A$ is a Lyapunov function for the
action $1-$form for $(H,d\theta+\beta)$.
\end{theorem}
\begin{proof}
The action $A$ for $(d\theta,H)$ varies as follows on $u\in
\mathcal{M}(x,y;d\theta+\beta,H)$,
$$
\begin{array}{lll}
-\partial_s A(u) & = &  \int_0^1 d\theta
(\partial_s u,F(u))\, dt\\
& = & \int_0^1 d\theta
(F(u),J_{\beta}F_{\beta}(u))\, dt\\
& = & \int_0^1 d\theta
(F(u),(J+\delta J)(F+\delta F)(u))\, dt\\
& \geq & \|F(u)\|^2 - \| \delta J \| \, \| F(u) \|^2 - \| \delta
F(u) \| \, \| F(u) \|- \|\delta F(u)\| \, \|\delta J\| \, \|
F(u) \| \\
& \geq & \left(1-\|\delta J\| - \frac{1}{3}- \frac{1}{3}\|\delta
J\|\right) \|F(u)\|^2,
\end{array}
$$
using Lemma \ref{Lemma Difference of F} in the last line.
\end{proof}
%
%
\subsection{A priori energy estimate}
\label{Subsection A priori energy estimate}
%
We now want an a priori energy estimate for all $u\in
\mathcal{M}(x,y;d\theta+\beta,H)$ when $\| \beta\|$ is small. The
key idea is to reparametrize the action $A$ by energy and then use
the Lyapunov inequality $\partial_s A(u) \leq -\frac{1}{2} \| F(u)
\|^2$ of Theorem \ref{Theorem Lyapunov property}.
Let $e(s)$ denote the energy up to $s$ calculated with respect to
$(d\theta+\beta,J_{\beta})$,
$$e(s) = \int_{-\infty}^s \int_0^1 |\partial_s u|_{\beta}^2 \, dt\, ds
=\int_{-\infty}^s \| \partial_s u\|_{\beta}^2 \, ds
$$
where $|\cdot|_{\beta}$ is the norm corresponding to the metric
$(d\theta+\beta)(\cdot,J_{\beta}\cdot)$, and $\|\cdot\|_{\beta}$
is the $L^2$ norm integrated over time.

\begin{theorem}\label{Theorem A priori energy estimate}
Let $\beta$ be as in Theorem \ref{Theorem Lyapunov property}. Then
there is a constant $k>1$ such that for all $u\in
\mathcal{M}(x,y;d\theta+\beta,H)$,
$$
E(u) \leq k (A(x)-A(y)).
$$
\end{theorem}
\begin{proof}
$\partial_s e=\| \partial_s u\|_{\beta}^2$ vanishes at $s$
precisely if $F_{\beta}(u)=0$. By ignoring those $s$ for which
$\partial_s e=0$, we can assume that $\partial_s e>0$. Let $s(e)$
be the inverse of the function $e(s)$. Then reparametrize the
trajectory $u$ by
$$
\tilde{u}(e,t) = u(s(e),t).
$$
Since $\displaystyle\partial_e s = \frac{1}{\| \partial_s u
\|_{\beta}^2}$, we deduce
$\displaystyle\partial_e (A\circ \tilde{u}) = \frac{\partial_s
A(u)}{\|
\partial_s u \|_{\beta}^2} = \frac{\partial_s A(u)}{\|
F_{\beta}(u)\|_{\beta}^2}.$

\noindent Now apply respectively Theorem \ref{Theorem Lyapunov
property}, Lemma \ref{Lemma Difference of F} and the equivalence
of the norms $\|\cdot\|$ and $\|\cdot\|_{\beta}$,
$$
\partial_e (A\circ \tilde{u})  \leq
\frac{- \| F(u) \|^2}{2 \| F_{\beta}(u)\|_{\beta}^2}
\leq  \frac{- \| F_{\beta}(u) \|^2}{\textrm{constant}\cdot
\|F_{\beta}(u)\|^2}\\
 =\co\! -  \frac{1}{k} .
$$
Integrate in $e$ over $(e(-\infty),e(\infty))=(0,E(u))$ to get
$A(y)-A(x) \leq (-1/k) E(u)$. By making $\|\beta\|$ sufficiently
small, one can actually make $k$ arbitrarily close to $1$.
\end{proof}
%
%
\subsection{Transversality for deformations}
\label{Subsection Transversality for deformations}
%
%
We now prove a general result which guarantees transversality for
a $1-$parameter deformation $\mathcal{G}$ of a map $\mathcal{F}$
for which transversality holds. We need a preliminary lemma.
\begin{lemma}\label{Lemma transversality lemma}
Let $L:B_1 \to B_2$ be a surjective bounded operator of Banach
spaces, and consider a perturbation $L+P_{\varepsilon}:B_1 \to
B_2$ where $P_{\varepsilon}$ is a bounded operator which depends
on a topological parameter $\varepsilon$, with $P_0=0$ and $\|
P_{\varepsilon} \| \to 0$ as $\varepsilon \to 0$.
\begin{enumerate}
\item If $L$ is Fredholm then so is $L+P_{\varepsilon}$ for small
$\varepsilon$.

\item If $L$ is Fredholm and surjective, then so is
$L+P_{\varepsilon}$ for small $\varepsilon$.
\end{enumerate}
\end{lemma}
\begin{proof}
The Fredholm property is a norm-open condition, hence (1). Recall
some general results relating an operator $L:B_1\to B_2$ to its
Banach dual $L^*:B_2^* \to B_1^*$:
\begin{enumerate}
\renewcommand{\labelenumi}{\roman{enumi})}
\item $L$ is surjective if and only if $L^*$ is injective and
$\textrm{im}\, L$ is closed;

\item if $L$ is Fredholm then $L^*$ is Fredholm;

\item a Fredholm operator is injective if and only if it is
bounded below.

\end{enumerate}
In (2), $L^*$ is bounded below, say $\| L^*v \| \geq \delta_L \|
v\|$ for all $v\in B_2^*$, so
$$\|(L+P_{\varepsilon})^*v\| \geq \| L^*v
\| - \| P_{\varepsilon}^* v \| \geq (\delta_L - \|
P_{\varepsilon}^*\| )\, \| v\|.$$
If $\varepsilon$ is so small that $\delta_L> \|
P_{\varepsilon}^*\|=\| P_{\varepsilon}\|$, then
$(L+P_{\varepsilon})^*$ is bounded below and so
$L+P_{\varepsilon}$ is surjective.
\end{proof}
\begin{theorem}\label{Theorem Transversality for Banach bundles}
Let $Y\to X$ be a Banach vector bundle. Suppose that a
differentiable section $\mathcal{F}:X\to Y$ is transverse to the
zero section with Fredholm differential $D_u\mathcal{F}$ at all
$u\in \mathcal{F}^{-1}(0)$. Let $\mathcal{S}:\R \times X \to Y$ be
a differentiable parameter-valued section with
$\mathcal{S}(0,\cdot)=0$. Then for the deformation
$\mathcal{G}=\mathcal{F}+\mathcal{S}:\R \times X \to Y$,
\begin{enumerate}
\item $\mathcal{G}^{-1}(0)$ is a smooth submanifold near
$\{0\}\times \mathcal{F}^{-1}(0)$;

\item $\mathcal{G}^{-1}(0)$ is transverse to $\{ \lambda = 0\} =\{
0 \} \times X$ $\mathrm{(}$where $\lambda$ is the
$\R$-coordinate$\mathrm{)}$;

\item If $0$ is an index zero regular value of $\mathcal{F}$ and
$\mathcal{G}^{-1}(0)$ is compact near $\lambda=0$, then the
deformation $\mathcal{G}^{-1}(0)$ of $\mathcal{F}^{-1}(0)$ is
trivial near $\lambda=0$,
$$
\mathcal{G}^{-1}(0) \cap \{ \lambda \in [-\lambda_0,\lambda_0]  \}
\cong [-\lambda_0,\lambda_0] \times \mathcal{F}^{-1}(0).
$$
\end{enumerate}
\end{theorem}
\begin{proof}

The first claim essentially follows from the implicit function
theorem and Lemma \ref{Lemma transversality lemma} applied to the
operators $L = D_u \mathcal{F}$ and $P_{\varepsilon} =
D_{\lambda,u} \mathcal{S}$ with parameter
$\varepsilon=(\lambda,u)$. More precisely, we reduce to the local
setup by choosing an open neighbourhood $U$ of $u$ so that $T_U X
\cong U \times B_1$, $T_U Y \cong U \times B_2$,
$$
T_{[-\lambda_0,\lambda_0] \times U} (\R \times X) \cong
([-\lambda_0,\lambda_0] \times U) \times \R \times B_1,
$$
so locally $D_u \mathcal{F} : B_1 \to B_2$ and $D_{\lambda,u}
\mathcal{S} : \R \times B_1 \to B_2$. Suppose
$\mathcal{F}(u_0)=0$, then apply Lemma \ref{Lemma transversality
lemma} to $L=D_{u_0}\mathcal{F}$ and $P_{(\lambda,u)} =
D_{u}\mathcal{F}-D_{u_0}\mathcal{F} + D_{(\lambda,u)}\mathcal{S}$.
Therefore $D_{\lambda,u} \mathcal{G}=L+P_{(\lambda,u)}$ is
Fredholm and surjective, so by the implicit function theorem
$\mathcal{G}^{-1}(0)$ is a smooth submanifold for $u$ close to
$u_0$. Thus claim (1) follows.

Observe that at $(\eta,\xi) \in T\R \oplus TX$,
$$
\begin{array}{lll}
D_{0,u}\mathcal{G}\cdot (\eta,\xi) & = & D_u \mathcal{F}\cdot \xi
+ D_{0,u} \mathcal{S}\cdot \xi +
\left.\partial_{\lambda}\right|_{\lambda=0} \mathcal{S} \cdot \eta \\
& = & D_u \mathcal{F}\cdot \xi +
\left.\partial_{\lambda}\right|_{\lambda=0} \mathcal{S} \cdot
\eta.
\end{array}
$$
Therefore, $D_{0,u}\mathcal{G}\cdot (0,\xi)=D_u \mathcal{F}\cdot
\xi$. We deduce that $\textrm{im} D_u \mathcal{F} \subset
\textrm{im} D_{0,u}\mathcal{G}$ and $\ker D_u \mathcal{F} \subset
\textrm{ker} D_{0,u}\mathcal{G}$. Since $D_u \mathcal{F}$ is
surjective whenever $\mathcal{F}(u)=0$ ($=\mathcal{G}(0,u)$), also
$D_{0,u} \mathcal{G}$ is surjective and therefore
$T_{0,u}\mathcal{G}^{-1}(0)\cong\ker\, D_{0,u} \mathcal{G}$ must
be $1$ dimension larger than $\ker D_u\mathcal{F}$, so it contains
some vector $(1,\xi)$, which implies claim (2).

This also relates the indices at solutions of $\mathcal{F}(u)=0$:
$$
\textrm{ind}\, D_{0,u}\mathcal{G} = \dim \ker\, D_{0,u}\mathcal{G}
=\dim \ker\, D_u\mathcal{F}+1 = \textrm{ind}\, D_u \mathcal{F} +1.
$$

If $0$ is an index zero regular value of $\mathcal{F}$, then
$\mathcal{F}^{-1}(0)$ is $0-$dimensional and $\mathcal{G}^{-1}(0)$
is a $1-$dimensional submanifold near $0\times
\mathcal{F}^{-1}(0)$ diffeomorphic to a product
$[-\lambda_0,\lambda_0] \times \mathcal{F}^{-1}(0)$, for some
small $\lambda_0$. If $\mathcal{G}^{-1}(0)$ is compact near
$\lambda=0$ then for sufficiently small $\lambda_0$ all solutions
of $\mathcal{G}(\lambda,u)=0$ with $|\lambda|\leq \lambda_0$ will
be close to $0\times \mathcal{F}^{-1}(0)$, proving claim (3).
\end{proof}
%
%
%
%
\subsection{The $1-$parameter family
of moduli spaces} \label{Subsection 1-parameter family}
%
%

Let $H$ be a Hamiltonian which is linear at infinity. In this
section we will prove
\begin{theorem}\label{Theorem Transversality of the 1-parameter
family} For $\beta$ as in Theorem \ref{Theorem Lyapunov property}
the family of moduli spaces
$$\mathcal{M}_{\lambda}(x,y) =
\mathcal{M}(x,y;d\theta+\lambda\beta,H)$$
is smoothly trivial near $\lambda=0$,
$$ \bigsqcup_{-\lambda_0 < \lambda < \lambda_0}
\mathcal{M}_{\lambda}(x,y) \cong \mathcal{M}(x,y;d\theta,H) \times
(-\lambda_0,\lambda_0).
$$
In particular, the identity map
$$
\textrm{id}:SC^*(H,d\theta+\lambda\beta) \to
SC^*(H,d\theta;\underline{\Lambda}_{dA+\lambda\tau\beta})
$$
is a chain isomorphism for all small $\lambda$, where $A(x)=-\int
x^*\theta + \int H(x)\, dt$ is the action functional for
$(H,d\theta)$.
\end{theorem}
\begin{proof} Let $X_{\lambda \beta}$ be the Hamiltonian vector field
determined by $(H,d\theta+\lambda \beta)$. We want to compare the
following two maps,
$$
\mathcal{F}(u)=\partial_s u + J(\partial_t u-X)
\quad\textrm{and}\quad \mathcal{G}(u)=\partial_s u + J_{\lambda
\beta}(\partial_t u-X_{\lambda \beta}),
$$
since $\mathcal{F}^{-1}(0) = \mathcal{M}(x,y)$ and
$\mathcal{G}^{-1}(0) = \cup_{\lambda}\mathcal{M}_{\lambda}(x,y)$.

These maps can be extended to sections $X \to Y$ of an appropriate
Banach vector bundle and generically $\mathcal{F}$ is a Fredholm
map (its linearizations are Fredholm operators). Indeed for $k\geq
1$ and $p>2$, we can take $Y$ to be the $W^{k-1,p}$ completion of
the space of smooth sections of $u^*TM$ with suitable exponential
decay at the ends. The base $X$ is the space of $W^{k,p}$ maps
$u:\R \times S^1 \to M$ connecting two fixed $1-$periodic
Hamiltonian orbits. We refer to Salamon \cite{Salamon} and
McDuff-Salamon \cite{McDuff-Salamon} for a precise description.

For convenience, denote $\delta J=J_{\lambda \beta}-J$ and $\delta
X = X - X_{\lambda \beta}$. We may assume that $\delta J$ is
$C^2-$small, and we showed in Lemma \ref{Lemma Difference of F}
that
$$|\delta X| \leq \frac{|\lambda|\, \| \beta \|}{1-|\lambda|\,
\| \beta \|} \, |X_{\supp \beta}|.$$
So $\delta J$, $\delta X$ are small for small $\lambda$. We can
rewrite
$\mathcal{G}(\lambda,u)=\mathcal{F}(u)+\mathcal{S}(\lambda,u)$,
where
$$
\mathcal{S}(\lambda,u) = \delta J \cdot (F(u)+\delta X)+J\delta X,
$$
where $F(u) = \partial_t u - X(u)$. $\mathcal{S}$ is supported at
those $(u,s,t)$ with $u(s,t)\in\supp \beta$, and $\mathcal{S}: X
\to Y$ is a differentiable parameter-valued section vanishing at
$\lambda=0$.

By the a priori energy estimate of Theorem \ref{Theorem A priori
energy estimate}, $\mathcal{G}^{-1}(0)$ is compact near
$\lambda=0$. Theorem \ref{Theorem Transversality for Banach
bundles} implies that if $0$ is an index zero regular value of
$\mathcal{F}$ then $\mathcal{G}^{-1}(0)$ is a trivial
$1-$dimensional family in the parameter $\lambda$, for small
$\lambda$.

Thus, for sufficiently small $\lambda_0$, there is a natural
bijection between the moduli spaces which define the differentials
of $SC^*(H, d\theta + \lambda_0\beta)$ and
$SC^*(H,d\theta;\underline{\Lambda}_{dA+\lambda_0\tau\beta})$.
Indeed, if $u_{\lambda_0} \in
\mathcal{M}_0(x,y;H,d\theta+\lambda_0\beta)$ then there is a
natural $1-$parameter family
$$u_{\lambda} \in
\mathcal{M}_0(x,y;H,d\theta+\lambda\beta)$$
connecting $u_{\lambda_0}$ to some $u_0 \in
\mathcal{M}_0(x,y;H,d\theta)$. Since $u_{\lambda_0}$ is homotopic
to $u_0$ relative endpoints via $u_{\lambda}$, the local system
$\underline{\Lambda}_{dA+\lambda_0\tau\beta}$ yields the same
isomorphism for $u_{\lambda_0}$ as for $u_0$, which is
multiplication by
$$
t^{-\int u^*d\theta + \int_0^1 (H(x)- H(y))\, dt - \int
u^*(\lambda_0\beta)} = t^{-\int u^*(d\theta + \lambda_0\beta) +
\int_0^1 (H(x)- H(y))\, dt}
$$
and which is the same weight used in the definition of $\partial
y$ for $SC^*(H,d\theta+\lambda_0\beta)$. Therefore the two
complexes have exactly the same generators and the same
differential, and in particular the identity map between them is a
chain isomorphism.
\end{proof}
%
%
%
%
\subsection{Continuation of the $1-$parameter family}
\label{Subsection Continuation of the 1-parameter family}
%
%

%
\begin{theorem}\label{Theorem Transversality of the 1-parameter
family continuation} Let $\beta$ be as in Theorem \ref{Theorem
Lyapunov property}. Let $H_s$ be a monotone homotopy. Then the
family of moduli spaces of parametrized Floer trajectories
$$\mathcal{M}_{\lambda}(x,y;H_s) =
\mathcal{M}(x,y;d\theta+\lambda\beta,H_s)$$
is smoothly trivial near $\lambda=0$. In particular, the following
diagram commutes for all small enough $\lambda$,
$$
\xymatrix{SC^*(H_+,d\theta+\lambda\beta)
\ar@{->}[r]^-{\textrm{id}}\ar@{->}[d]_{\textrm{continuation}} &
SC^*(H_+,d\theta;\underline{\Lambda}_{dA+\tau\lambda\beta})
\ar@{->}[d]^{\textrm{continuation}} \\
SC^*(H_-,d\theta+\lambda\beta) \ar@{->}[r]^-{\textrm{id}} &
SC^*(H_-,d\theta;\underline{\Lambda}_{dA+\tau\lambda\beta}) }
$$
\end{theorem}
\begin{proof}
Let $X_{s,\lambda \beta}$ be the Hamiltonian vector field
determined by $(H_s,d\theta+\lambda \beta)$, and let
$X_s=X_{s,0}$. The claim follows by mimicking the proof of Theorem
\ref{Theorem Transversality of the 1-parameter family} for
$$
\mathcal{F}(u)=\partial_s u + J_s(\partial_t u-X_s)
\quad\textrm{and}\quad \mathcal{G}(u)=\partial_s u + J_{s,\lambda
\beta}(\partial_t u-X_{s,\lambda \beta}). \qedhere
$$
\end{proof}
\begin{theorem}\label{Theorem Transversality of the 1-parameter
family continuation 2} Let $\beta$ be as in Theorem \ref{Theorem
Lyapunov property}. Let $\lambda$ be so small that Theorem
\ref{Theorem Transversality of the 1-parameter family} holds for
$H$. Let $\varphi_{\varepsilon}$ be a smooth parameter-valued
isotopy of $\widehat{M}$, with $\varphi_0=\textrm{id}$, such that
$\varphi_{\varepsilon}^*H$ is a monotone homotopy in
$\varepsilon$. Let $H_{s,\varepsilon}=\varphi_{s\varepsilon}^*H$
for $s\in [0,1]$ be the homotopy from $H$ to
$\varphi_{\varepsilon}^*H$. Then the family of moduli spaces of
parametrized Floer trajectories
$\mathcal{M}_{\varepsilon}(x,y;H_{s,\varepsilon}) =
\mathcal{M}(x,y;d\theta+\lambda\beta,H_{s,\varepsilon})$ is
smoothly trivial near $\varepsilon=0$. So there is a commutative
diagram of chain isomorphisms for all small $\varepsilon$,
$$
\xymatrix{SC^*(H,d\theta+\lambda\beta)
\ar@{->}[r]^-{\textrm{id}}\ar@{->}[d]_{\textrm{continuation}} &
SC^*(H,d\theta;\underline{\Lambda}_{dA+\tau\lambda\beta})
\ar@{->}[d]^{\textrm{continuation}} \\
SC^*(\varphi_{\varepsilon}^*H,d\theta+\lambda\beta)
\ar@{->}[r]^-{\textrm{id}} &
SC^*(\varphi_{\varepsilon}^*H,d\theta;\underline{\Lambda}_{dA+\tau\lambda\beta})
}
$$
where the vertical maps send the generators $x\mapsto
\varphi_{\varepsilon}^{-1}(x)$.
\end{theorem}
\begin{proof}
Let $X_{s,\varepsilon}$ be the Hamiltonian vector field determined
by $(H_{s,\varepsilon},d\theta+\lambda \beta)$, and let
$X=X_s=X_{s,0}$. The claim follows by mimicking the proof of
Theorem \ref{Theorem Transversality of the 1-parameter family} for
$$
\mathcal{F}(u)=\partial_s u + J_s(\partial_t u-X)
\quad\textrm{and}\quad \mathcal{G}(u)=\partial_s u +
J_{s,\varepsilon}(\partial_t u-X_{s,\varepsilon}). \qedhere
$$
\end{proof}
%
%
%
%
\subsection{Construction of the isomorphism}
\label{Subsection Construction of the isomorphism}
%
%

We now give the proof outlined in \ref{Subsection Outline of the
argument}.

Let $\beta$ be a closed two-form compactly supported in the
interior of $M$, and suppose that $d\theta+s\beta$ is symplectic
for all $0\leq s \leq 1$ (so that Lemma \ref{Lemma deforming the
sympl form} applies).

Let $H^m$ be a Hamiltonian linear at infinity with slope $m$. Up
to a continuation isomorphism on symplectic cohomologies, we may
assume that all critical points of $H^m$ in the interior of $M$
lie in a neighbourhood $V$ contained in $M \setminus \supp\beta$.
This technical remark is explained in section \ref{Subsection
Technical remark}.
Define $\psi_{\mu}^m$ by the diagram of isomorphisms
$$
\xymatrix{SC^*(H^m,d\theta+\beta)
\ar@{->}[d]_{(1)}^{\textrm{Liouville } \varphi_{\mu}}
\ar@{-->}[r]^{\psi_{\mu}^m} &
SC^*(H^m,d\theta;\underline{\Lambda}_{\tau\beta})
\ar@{->}[ddd]_{(6)}^{\textrm{rescale}} \\
SC^*(\varphi_{\mu}^*H^m,\mu d\theta + \varphi_{\mu}^*\beta)
\ar@{->}[d]_{(2)}^{\textrm{Moser } \sigma_{\mu}} &\\
SC^*(\phi_{\mu}^*H^m,\mu d\theta + \beta)
\ar@{->}[d]_{(3)}^{\textrm{continuation}} &
\\
SC^*(\mu H^m,\mu d\theta + \beta)
\ar@{->}[d]_{(4)}^{\textrm{rescale}} & SC^*(H^m,
d\theta;\underline{\Lambda}_{\mu^{-1}\tau\beta})
\ar@{->}[d]_{(7)}^{\textrm{change of basis}}
\\
SC^*(H^m, d\theta + \mu^{-1}\beta)
\ar@{->}[r]^-{\textrm{id}}_-{(5)} & SC^*(H^m,
d\theta;\underline{\Lambda}_{dA+\mu^{-1}\tau\beta}) }
$$
where the maps are defined as follows:
\begin{enumerate}
\item apply $\varphi_{\mu}$, the Liouville flow for time $\log
\mu$ (see \ref{Subsection Symplectic manifolds with contact type
boundary} for the definition of the Liouville vector field);

\item apply the Moser symplectomorphism
$\sigma_{\mu}:(\widehat{M},\mu\, d\theta + \beta) \to
(\widehat{M},\mu \, d\theta+\varphi_{\mu}^*\beta)$ obtained by
Lemmas \ref{Lemma deforming the sympl form} and \ref{Lemma Moser
deformation}, and denote
$\phi_{\mu}=\sigma_{\mu}\circ\varphi_{\mu}$;

\item observe that $\phi_{\mu}^*H^m$ has slope $\mu m$ at
infinity, so the linear interpolation from $\mu H^m$ to
$\phi_{\mu}^* H^m$ is a compactly supported homotopy and therefore
induces a continuation isomorphism;

\item metric rescaling by $\mu^{-1}$ (Lemma \ref{Lemma Metric
rescaling}), which changes $t$ to $T=t^{(\mu^{-1})}$;

\item the identity map is a chain isomorphism by Theorem
\ref{Theorem Transversality of the 1-parameter family} provided
$\mu$ is sufficiently large (depending on $m$);

\item rescale $\tau \beta$ to $\mu^{-1} \tau \beta$, so change $t$
to $T=t^{(\mu^{-1})}$;

\item adding an exact form $dA$ to $\mu^{-1}\tau \beta$, where $A$
is the action $1-$form for $(H^m,d\theta)$, corresponds to a
change of basis $x\mapsto T^{A(x)} x$ by Lemma \ref{Lemma Moser
deformation twisted}.
\end{enumerate}

\begin{lemma}\label{Lemma Tau maps independent of mu}
The map $\psi_{\mu}^m: SH^*(H^m,d\theta+\beta) \to SH^*(H^m,
d\theta; \underline{\Lambda}_{\tau\beta})$ on homology does not
depend on the choice of large $\mu$.
\end{lemma}
\begin{proof}
In this proof we abbreviate $H^m$ by $H$ and pullbacks $\phi^*$ by
$\phi$. Consider $\mu'$ close to $\mu$, and write $ \phi =
\phi_{\mu'}\phi_{\mu}^{-1}$ and $\varphi =
\varphi_{\mu'}\varphi_{\mu}^{-1}$. Observe the following
commutative diagram, in which the top row and bottom diagonal are
part of the construction of the maps $\psi_{\mu}^m$ and
$\psi_{\mu'}^m$, for $\mu'>\mu$.
$$
\xymatrix{ SH^*(H,d\theta+\beta)
\ar@{->}[r]^-{\phi_{\mu}}\ar@{->}[dr]_-{\phi_{\mu'}} &
SH^*(\phi_{\mu} H,\mu\, d\theta+\beta) \ar@{->}[d]^{\phi^{-1}}
\ar@{->}[r]^{\textrm{continu.}}\ar@{->}[dr]^{\;\quad\textrm{continuation}}
& SH^*(\mu H, \mu \,
d\theta + \beta) \ar@{->}[d]^{\textrm{continuation}} \\
& SH^*(\phi_{\mu'} H,\mu'\,d\theta+\beta)
\ar@{->}[dr]_{\textrm{continuation}\quad} &
SH^*(\phi^{-1}\mu' H, \mu \, d\theta + \beta) \ar@{->}[d]^{\phi^{-1}}\\
& & SH^*(\mu' H, \mu' \, d\theta + \beta) }
$$

The last vertical composite, after a metric rescaling, is the map
$$ \phi^{-1} \circ C: SH^*(H, d\theta + \mu^{-1}\beta) \to
SH^*(H, d\theta + {\mu'}^{-1}\beta) $$
where $C$ is the continuation map
$$C:
SH^*(H, d\theta + \mu^{-1}\beta) \to SH^*({\mu}^{-1}\phi^{-1} \mu'
H, d\theta + {\mu}^{-1}\beta).$$
For $\mu'$ sufficiently close to $\mu$, $\phi^{-1}$ is an isotopy
of $\widehat{M}$ close to the identity, therefore by Theorem
\ref{Theorem Transversality of the 1-parameter family continuation
2}, $C$ maps the generators by $\phi$. Thus $\phi^{-1} \circ C =
\textrm{id}$ for $\mu'$ close to $\mu$.

For the twisted symplectic cohomology we just apply changes of
basis so we deduce the following commutative diagram (using
abbreviated notation),
$$
\xymatrix{ SH^*(d\theta+\beta) \ar@{->}[r] \ar@{->}[dr] & SH^*(
d\theta + \mu^{-1} \beta) \ar@{->}[d]^{\textrm{id}}
\ar@{->}[r]^-{\textrm{id}} &
SH^*(\underline{\Lambda}_{dA+\mu^{-1}\tau\beta})
\ar@{->}[d]^-{\textrm{id}} \ar@{->}[r] &
SH^*(\underline{\Lambda}_{\tau\beta})
\\
& SH^*(d\theta + (\mu')^{-1}\beta) \ar@{->}[r]^-{\textrm{id}} &
SH^*(\underline{\Lambda}_{dA+(\mu')^{-1}\tau\beta}) \ar@{->}[ur]
 }
$$
We showed that this diagram holds for all $\mu'$ close to $\mu$.
Suppose it holds for all $\mu,\mu' \in [\mu_0,\mu_1)$, for some
maximal such $\mu_1<\infty$. Apply the above result to $\mu =
\mu_1$, then the diagram holds for all $\mu,\mu' \in
(\mu_1-\epsilon,\mu_1+\epsilon)$, for some $\epsilon>0$. Thus it
holds for all $\mu,\mu'\in [\mu_0,\mu_1+\epsilon)$. So there is no
maximal such $\mu_1$ and the diagram must hold for all large
enough $\mu,\mu'$, and thus the map $SH^*(H,d\theta+\beta) \to
SH^*(H,d\theta;\underline{\Lambda}_{\tau\beta})$ does not depend
on the choice of (large) $\mu$.
\end{proof}
\begin{lemma}\label{Lemma Tau maps commute with continuation}
The maps $\psi^m: SH^*(H^m,d\theta+\beta) \to SH^*(H^m, d\theta;
\underline{\Lambda}_{\tau\beta})$ commute with the continuation
maps induced by monotone homotopies.
\end{lemma}
\begin{proof}
Let $H_s$ be a monotone homotopy from $H^{m'}$ to $H^m$. By
theorem \ref{Theorem Transversality of the 1-parameter family
continuation}, for sufficiently large $\mu$ the following diagram
commutes
$$
\xymatrix{SC^*(H_+,d\theta+\mu^{-1}\beta)
\ar@{->}[r]^-{\textrm{id}}\ar@{->}[d]_{\textrm{continuation}} &
SC^*(H_+,d\theta;\underline{\Lambda}_{dA+\mu^{-1}\tau\beta})
\ar@{->}[d]^{\textrm{continuation}} \\
SC^*(H_-,d\theta+\mu^{-1}\beta) \ar@{->}[r]^-{\textrm{id}} &
SC^*(H_-,d\theta;\underline{\Lambda}_{dA+\mu^{-1}\tau\beta}) }
$$
and by Lemma \ref{Lemma Moser deformation twisted} we deduce the
required commutative diagram
$$
\xymatrix{SC^*(H_+,d\theta+\beta)
\ar@{->}[r]^{\psi^m}\ar@{->}[d]_{\textrm{continuation}} &
SC^*(H_+,d\theta;\underline{\Lambda}_{\tau\beta})
\ar@{->}[d]^{\textrm{continuation}} \\
SC^*(H_-,d\theta+\beta) \ar@{->}[r]^{\psi^{m'}} &
SC^*(H_-,d\theta;\underline{\Lambda}_{\tau\beta})}
$$
\renewcommand{\qed}{}
\end{proof}
\begin{theorem}\label{Theorem Lim of tau twisting maps}
Let $\beta$ be a closed two-form compactly supported in the
interior of $M$, and suppose that $d\theta+s\beta$ is symplectic
for $0\leq s \leq 1$. Then there is an isomorphism
$$\psi: SH^*(M,d\theta+\beta)
\to SH^*(M,d\theta;\underline{\Lambda}_{\tau\beta}).$$
\end{theorem}
\begin{proof}
By Lemma \ref{Lemma Tau maps independent of mu} the map
$\psi^m=\psi_{\mu}^m$ on homology is independent of $\mu$ for
large $\mu$, and by Lemma \ref{Lemma Tau maps commute with
continuation} the maps $\psi^m$ commute with continuation maps.
The direct limit is an exact functor, so $\psi=\varinjlim \psi^m$
is an isomorphism.
\end{proof}
\begin{remark}
The theorem can sometimes be applied to deformations $\omega_s$
which are not compactly supported by using Gray's stability
theorem e.g. see Lemma \ref{Lemma SH infinitesimal same as
nonexact ADE}.
\end{remark}
\begin{remark}
Let $\beta \in H^2(M;\R)$ come from $H^2(\partial M;\R)$ by the
Thom construction. Then $SH^*(M,d\theta+\beta)\cong
SH^*(M,d\theta;\Lambda)$, the ordinary symplectic cohomology with
underlying ring $\Lambda$. Indeed, suppose $\beta$ vanishes on
$$\textrm{Fix}\,(\varphi_{\mu}) = \lim_{\mu \to -\infty} \varphi_{\mu}(M).$$
Let $H=h(R)$ be a convex Hamiltonian defined in a neighbourhood
$R< R_0$ of $\textrm{Fix}\,(\varphi_{\mu})$ where $\beta$
vanishes, such that $h'(R)\to \infty$ as $R\to R_0$. Let $H^m=h$
if $h'\leq m$ and let $H^m$ be linear with slope $m$ elsewhere.
Then the Floer solutions concerned in the symplectic chain groups
will all lie in the subset of $M$ where $\beta=0$.
\end{remark}
%
%
\subsection{Technical remark}
\label{Subsection Technical remark}
%
%
We assumed in \ref{Subsection Construction of the isomorphism}
that all critical points of $H$ in the interior of $M$ lie in a
neighbourhood $V\subset M \setminus \supp\beta$. We can do this as
follows.

Pick a small neighbourhood $V$ around $\textrm{Crit}(H)$ so that
$\beta|_V = d\alpha$ is exact. We may assume that $\alpha$ is
supported in $V$. To construct the isomorphism of \ref{Subsection
Construction of the isomorphism} we need to homotope
$d\theta+\mu^{-1}\beta$ to $d\theta+\mu^{-1}(\beta-d\alpha)$ for
all large $\mu$. This can be done by a Moser isotopy compactly
supported in $V$ via the exact deformation
$\omega_s=d\theta+\mu^{-1}(\beta-sd\alpha)$. Since $\partial_s
\omega_s = -\mu^{-1}d\alpha$, for large $\mu$ the Moser isotopy
$\phi_s$ is close to the identity. Therefore during the isotopy
the critical points of $\phi_s^*H$ stay within $V$. This
guarantees that the Palais-Smale Lemma \ref{Corollary
Palais-Smale} can be applied for $V$ independently of large $\mu$,
and the construction \ref{Subsection Construction of the
isomorphism} can be carried out with minor modifications.
%
%
%
%
%
%
%
%

\section{ALE spaces} \label{Section ADE spaces}
%
%
%
%
\subsection{Hyperk\"{a}hler manifolds}
\label{Subsection Hyperkahler manifolds}
%
We suggest \cite{Hitchin} for a detailed account of
Hyperk\"{a}hler manifolds and ALE spaces.

Recall that a symplectic manifold $(M,\omega)$ is
\emph{K\"{a}hler} if there is an integrable $\omega-$compatible
almost complex structure $I$. Equivalently, a Riemannian manifold
$(M,g)$ is K\"{a}hler if there is an orthogonal almost complex
structure $I$ which is covariant constant with respect to the
Levi-Civita connection. $(M,g)$ is called \emph{hyperk\"{a}hler}
if there are three orthogonal covariant constant almost complex
structures $I,J,K$ satisfying the quaternion relation $IJK=-1$.

The hyperk\"{a}hler manifold $(M,g)$ is therefore K\"{a}hler with
respect to each of the (integrable) complex structures $I,J,K$,
with corresponding K\"{a}hler forms
$$
\omega_I = g(I\cdot,\cdot),\quad \omega_J = g(J\cdot,\cdot),\quad
\omega_K = g(K\cdot,\cdot).
$$

Indeed, there is an $S^2$ worth of K\"{a}hler forms: any vector
$(u_I,u_J,u_K)\in S^2 \subset \R^3$ gives rise to a complex
structure $I_u=u_I I +u_J J+u_K K$ and a K\"{a}hler form
$$\omega_u=u_I \omega_I+u_J \omega_J+u_K \omega_K.$$

We will always think of $M$ as a complex manifold with respect to
$I$, and we recall from \cite{Hitchin2} that $\omega_J + i
\omega_K$ is a holomorphic symplectic structure on $M$ (a
non-degenerate closed holomorphic $(2,0)$ form). The form
$\omega_J+i\omega_K$ determines a trivialization of the canonical
bundle $\Lambda^{2,0}T^*M$, so $c_1(M)=0$ and the Conley-Zehnder
indices give a $\Z-$grading on symplectic cohomology (see
\ref{Subsection Maslov index and Conley-Zehnder index}).

\begin{lemma}\label{Lemma Lagr submfds in Hyperkahler}
Let $L\subset \mathbb{H}$ be an $I-$complex vector subspace of the
space of quaternions with $\dim_{\R}L=2$. Then $L$ is a real
Lagrangian subspace with respect to $\omega_J$ and $\omega_K$, and
a symplectic subspace with respect to $\omega_I$. After an
automorphism of $\mathbb{H}$, $L$ is identified with $\C\oplus 0
\subset \mathbb{H}$.
\end{lemma}
\begin{proof}
$L$ is a complex 1-dimensional subspace of $(\mathbb{H},I)$, so
$L$ is complex Lagrangian with respect to the $I$-holomorphic
symplectic form $\omega_c=\omega_J + i\omega_K$. Thus $L$ is a
real Lagrangian vector subspace of $\mathbb{H}$ with respect to
$\omega_J$ and $\omega_K$.

Moreover, given any vector $e_1 \in L$, let $e_2=Ie_1$, $e_3=Je_1$
and $e_4=Ke_1$. Then $L=\textrm{span}\{e_1,e_2\}$ and
$\omega_I(e_1,e_2)=g(e_2,e_2)>0$, so $L$ is symplectic with
respect to $\omega_I$ and corresponds to $\C\oplus 0$ in the
hyperk\"{a}hler basis $e_1,\ldots,e_4$.
\end{proof}

\subsection{Hyperk\"{a}hler quotients}
\label{Subsection Hyperkahler quotients}
%
Let $M$ be a simply connected hyperk\"{a}hler manifold. Let $G$ be
a compact Lie group $G$ acting on $M$ and preserving $g,I,J,K$.
Then corresponding to the forms $\omega_I$, $\omega_J$, $\omega_K$
there exist moment maps $\mu_I$, $\mu_J$, $\mu_K$. Recall that if
$\zeta$ is in the Lie algebra $\mathfrak{g}$ of $G$, then it
generates a vector field $X_{\zeta}$ on $M$. A moment map $\mu:M
\to \mathfrak{g}^{\vee}$ is a $G-$equivariant map such that
$$d\mu_m (\zeta) = \omega(X_\zeta(m),\cdot) \;\textrm{ at }\; m\in M.$$
For simply connected $M$ such a $\mu$ exists and is determined up
to the addition of an element in $Z=(\mathfrak{g}^{\vee})^G$, the
invariant elements of the dual Lie algebra $\mathfrak{g}^{\vee}$.

Putting these moment maps together yields $\mu =
(\mu_I,\mu_J,\mu_K): M \to \R^3\otimes \mathfrak{g}^{\vee}, $
and for $\zeta \in \R^3\otimes Z$ we may define the
hyperk\"{a}hler quotient space
$$
X_{\zeta} = \mu^{-1}(\zeta)/F.
$$

If $F$ acts freely on $\mu^{-1}(\zeta)$ then $X_{\zeta}$ is a
smooth manifold of dimension $\dim M - 4 \dim F$ and the
structures $g,I,J,K$ descend to $X_{\zeta}$ making it
hyperk\"{a}hler (see \cite{Hitchin2}).

%
\subsection{ALE spaces and ADE singularities}
\label{ALE and ADE spaces}
%

\begin{definition}
Let $\Gamma$ be any finite subgroup of $SU(2)$ (or, equivalently,
$SL_2(\C)$). An \emph{ALE space} (asymptotically locally
Euclidean) is a hyperk\"{a}hler $4-$manifold with precisely one
end which at infinity is isometric to the quotient $\C^2/\Gamma$,
where $\C^2/\Gamma$ is endowed with a metric that differs from the
Euclidean metric by order $\mathcal{O}(r^{-4})$ terms and which
has the appropriate decay in the derivatives.
\end{definition}

Kronheimer proved in \cite{Kronheimer2} that ALE spaces are
(particularly nice) models for the minimal resolution of the
quotient singularities $\C^2/\Gamma$. More precisely, any ALE
space is diffeomorphic to such a minimal resolution, and
vice-versa.

We now recall Kronheimer's construction \cite{Kronheimer} of ALE
spaces as hyperk\"{a}hler quotients. Let $R$ be the left regular
representation of $\Gamma \subset SU(2)$ endowed with the natural
Euclidean metric,
$$R=\bigoplus_{\gamma\in \Gamma} \C e_{\gamma} \cong \C^{|\Gamma|}.$$
Denoting by $\C^2$ the natural left $SU(2)-$module, let
$$M=(\C^2 \otimes_{\C} \textrm{Hom}_{\C}(R,R))^{\Gamma}$$
be the pairs of endomorphisms $(\alpha,\beta)$ of $R$, which are
invariant under the induced left action of $\Gamma$. We make $M$
into a hyperk\"{a}hler vector space by letting $I$ act by $i$ and
$J$ by $J(\alpha,\beta)=(-\beta^*,\alpha^*)$.

The Lie group
$$F=\textrm{Aut}_{\C}(R,R)^{\Gamma}/\{\textrm{scalar maps}\}$$
of unitary automorphisms of $R$ which are $\Gamma-$invariant act
by conjugation on $M$, $f\cdot (\alpha,\beta)=(f\alpha f^{-1},
f\beta f^{-1})$, where we quotiented by the scalar matrices since
they act trivially. The corresponding Lie algebra $\mathfrak{f}$
corresponds to the traceless elements of $\textrm{Hom}_{\C}(R,R)$,
and the moment maps are:
$$
\mu_I(\alpha,\beta)=\frac{1}{2} i
([\alpha,\alpha^*]+[\beta,\beta^*]), \quad
(\mu_J+i\mu_K)(\alpha,\beta)=[\alpha,\beta].
$$

By McKay's correspondence, this description can be made explicit.
Recall that $R=\oplus n_i R_i$, where the $R_i$ are the complex
irreducible representations of $\Gamma$ of complex dimension
$n_i$. Then $\C^2 \otimes R_i \cong \oplus_j A_{ij} R_j$ where $A$
is the adjacency matrix describing an extended Dynkin diagram of
ADE type (the correspondence between $\Gamma$ and the type of
diagram is described in the Introduction). It follows that
$$
M = \bigoplus_{i\to j} \textrm{Hom}(\C^{n_i},\C^{n_j})
$$
where each edge $i\to j$ of the extended Dynkin diagram appears
twice, once for each choice of orientation. Moreover,
$$
F = (\oplus_{i} U(n_i)) / \{\textrm{scalar maps}\}
$$
where the unitary group $U(n_i)$ acts naturally on $\C^{n_i}$.

The hyperk\"{a}hler quotient for $\zeta\in
Z=\textrm{centre}(\mathfrak{f}^{\vee})$ is therefore
$$
X_{\zeta} = \mu^{-1}(\zeta)/F.
$$

\begin{definition}
Let $\mathfrak{h}_{\R}$ denote the real Cartan algebra associated
to the Dynkin diagram for $\Gamma$. Let the hyperplanes
$D_{\theta}= \ker \theta$ denote the walls of the Weyl chambers,
where the $\theta$ are the roots. We identify the centre $Z$ with
$\mathfrak{h}_{\R}$ by dualizing the map
$$
\textrm{centre}(\mathfrak{f}) \to \mathfrak{h}_{\R}^{\vee}, \; i
\, \pi_k \mapsto n_k \theta_k,
$$
where $\pi_k: R \to \C^{n_k} \otimes R_k$ are the projections to
the summands.

We call $\zeta \in \R^3 \otimes Z$ \emph{generic} if it does not
lie in $\R^3 \otimes D_{\theta}$ for any root $\theta$, i.e.
$\theta(\zeta_1)$, $\theta(\zeta_2)$, $\theta(\zeta_3)$ are not
all zero for any root $\theta$.
\end{definition}

\begin{theorem}[Kronheimer, \cite{Kronheimer}]\label{Theorem
Kronheimer} Let $\zeta \in \R^3 \otimes Z$ be generic. Then
$X_{\zeta}$ is a smooth hyperk\"{a}hler four-manifold with the
following properties.
\begin{enumerate}
\item $X_{\zeta}$ is a continuous family of hyperk\"{a}hler
manifolds in the parameter $\zeta$;

\item $X_0$ is isometric to $\C^2/\Gamma$;

\item there is a map $\pi:X_{\zeta} \to X_0$ which is an
$I-$holomorphic minimal resolution of $\C^2/\Gamma$, and $\pi$
varies continuously with $\zeta$;

\item \label{Item holo min resolution} in particular, $\pi$ is a
biholomorphism away from $\pi^{-1}(0)$ and $\pi^{-1}(0)$ consists
of a collection of $I-$holomorphic spheres with self-intersection
$-2$ which intersect transversely according to the Dynkin diagram
from $\Gamma$;

\item \label{Item Kronheimer thm forms} $H^2(X_{\zeta};\R) \cong
Z$ such that
$[\omega_I], [\omega_J], [\omega_K]$ map to
$\zeta_1,\zeta_2,\zeta_3.$

\item \label{Item Kronheimer H2} $H_2(X_{\zeta};\Z) \cong \{
\textrm{root lattice for } \Gamma \}$, such that the classes
$\Sigma$ with self-intersection $-2$ correspond to the roots;

\item $X_{\zeta}$ and $X_{\zeta'}$ are isometric hyperk\"{a}hler
manifolds if $\zeta,\zeta'$ lie in the same orbit of the Weyl
group;

\item Every ALE space asymptotic to $\C^2/\Gamma$ is isomorphic to
$X_{\zeta}$ for some generic $\zeta$.
\end{enumerate}
\end{theorem}
%
%
%
\subsection{Plumbing construction of ALE spaces}
\label{Subsection Plumbing construction of ADE spaces}
%
Our goal is to prove that for any ALE space $X$,
$SH^*(X;\omega)=0$ for a generic choice of (non-exact) symplectic
form $\omega$. By Theorem \ref{Theorem Kronheimer}.(\ref{Item
Kronheimer thm forms}) the cohomology class $[\omega_I]$ ranges
linearly in $\zeta_1$ over all of $H^2(X;\R)$. Therefore it
suffices to consider the hyperk\"{a}hler quotient $X=X_{\zeta}$
for all generic $\zeta=(\zeta_1,0,0) \in Z\otimes \R^3$.

\begin{lemma}\label{Lemma exceptional divisors are Lagr}
The exceptional divisors in $X$ are exact Lagrangian spheres with
respect to $\omega_J$ and $\omega_K$ and they are symplectic
spheres with respect to $\omega_I$. Moreover, the areas
$<\omega_I,\Sigma_m>$ of the exceptional spheres $\Sigma_m$ range
linearly in $\zeta_1$ over all possible positive values.
\end{lemma}
\begin{proof}
The first statement is an immediate consequence of Lemma
\ref{Lemma Lagr submfds in Hyperkahler}, using the fact that the
exceptional divisors in $X$ are holomorphic spheres by Theorem
\ref{Theorem Kronheimer}.(\ref{Item holo min resolution}). Note
that if a sphere is Lagrangian then it is exact since
$H^1(S^2;\R)=0$. The second statement is immediate since
$[\omega_I]$ ranges linearly in $\zeta_1$ over $H^2(X;\R)$ and the
$\Sigma_m$ generate $H_2(X;\Z)$, by Theorem \ref{Theorem
Kronheimer}.(\ref{Item Kronheimer H2}).
\end{proof}

The space $(X,\omega_J)$ is the plumbing of copies of $T^*\C P^1$,
plumbed according to the Dynkin diagram for $\Gamma$. Indeed, by
mimicking the proof of Weinstein's Lagrangian neighbourhood
theorem, one observes that a neighbourhood of the collection of
exceptional Lagrangian spheres is symplectomorphic to a plumbing
of copies of small disc cotangent bundles $DT^*\C P^1$. That
neighbourhood can be chosen so that its complement is a symplectic
collar diffeomorphic to $(S^3/\Gamma)\times [1,\infty)$, since $X$
is biholomorphic to $\C^2/\Gamma$ away from $0$.

\begin{remark} \label{Remark Plumbing or ALE does not matter}
We will show that the exact Lagrangians inside an ALE space $X$
must be spheres. To prove that this holds also for the above
plumbing, we don't actually need to know that the plumbing $Y$ is
all of $X$, the embedding $Y\hookrightarrow X$ provided by
Weinstein's theorem is enough. Indeed, the argument relies on
contradicting Corollary \ref{Corollary Functoriality trick} by
showing that $c_*1$ maps to $0$ via
$SH^*(Y,\omega_J;\underline{\Lambda}_{\tau \omega_I}) \to
H_{n-*}(\mathcal{L}L)\otimes \Lambda$. This is true since $c_*1$
is in the image of $SH^*(X,\omega_J;\underline{\Lambda}_{\tau
\omega_I}) \to SH^*(Y,\omega_J;\underline{\Lambda}_{\tau
\omega_I})$ by Theorem \ref{TheoremViterboTwistedFunctoriality},
and we will prove that $SH^*(X,\omega_J;\underline{\Lambda}_{\tau
\omega_I})=0$.
\end{remark}
%
%
\subsection{Contact hypersurfaces inside ALE spaces}
\label{Subsection Contact hypersurface inside ADE spaces}
%

\begin{lemma}\label{Lemma ADE contact type} Recall that any
$(u_I,u_J,u_K)\in S^2\subset \R^3$ gives rise to a K\"{a}hler form
$$\omega_u = u_I
\omega_I + u_J \omega_J + u_K \omega_K.$$

Then $(X,\omega_u)$ is a symplectic manifold such that
$\pi^{-1}(S_r^3/\Gamma)$ is a contact hypersurface in $X$ for all
sufficiently large $r$, so that $X$ can be thought of as a
symplectic manifold with contact type boundary with an infinite
collar attached. Moreover, $X$ is exact symplectic precisely when
$u_I=0$.
\end{lemma}
\begin{proof}
Recall $\pi: X \to \C^2/\Gamma$ denotes the resolution. Let
$\omega_u'$ denote the corresponding combination of forms for
$\C^2/\Gamma=\mathbb{H}/\Gamma$. On $\C^2/\Gamma$ the Liouville
vector field for any $\omega_u'$ is $Z=\partial_r$, and $\omega_u'
= d \theta_u'$ where $\theta_u' = i_Z\omega_u'$. Restricted to any
sphere $S_r^3/\Gamma$ of radius $r>0$, $\theta_u'$ is the
corresponding contact one-form.

By Theorem \ref{Theorem Kronheimer}, $X$ is asymptotic to
$\C^2/\Gamma = \mathbb{H}/\Gamma$ at infinity, therefore on
$\pi^{-1}(S_r^3/\Gamma)$, $\omega_u=d\theta_u$ where $\theta_u$
can be chosen to be asymptotic to $\theta_u'$. In particular,
since $\theta_u'\wedge d\theta_u'>0$ also $\theta_u\wedge
d\theta_u>0$ on $\pi^{-1}(S_r^3/\Gamma)$ for large $r$. Thus $X$
can be thought of as a contact type manifold with boundary
$\pi^{-1}(S_r^3/\Gamma)$ with the infinite collar
$\pi^{-1}(\cup_{\rho\geq r} S_{\rho}^3/\Gamma)$ attached. The last
statement follows by Lemma \ref{Lemma exceptional divisors are
Lagr}.
\end{proof}
%
%
\subsection{An $S^1-$action on ALE spaces}
\label{Subsection S1 action on ADE spaces}
%

Let $X=X_{\zeta_1,0,0}$ for generic $(\zeta_1,0,0)$. The
resolution $\pi: X \to \C^2/\Gamma$ can be described explicitly as
follows (following \cite{Hitchin}). The moment map equations are
$[\alpha,\beta]=0$ and $[\alpha,\alpha^*]+[\beta,\beta^*] =
-2i\zeta_1$. Since $\alpha,\beta$ commute by the first equation,
they have a common eigenvector $e$, say $(\alpha,\beta) e = (a,b)
e$. By $\Gamma-$invariance, $e^{\gamma} = R(\gamma)\cdot e$ is
also a common eigenvector such that
$$
(\alpha,\beta) e^{\gamma} = (\gamma\cdot (a,b)) e^{\gamma}.
$$

The map $X \to \C^2/\Gamma$, $(\alpha,\beta) \mapsto \Gamma\cdot
(a,b)$ is then an $I-$holomorphic minimal resolution. In fact
$\pi$ is also compatible with $J$ and $K$ if we identify
$\C^2/\Gamma = \mathbb{H}/\Gamma$.

\begin{theorem}\label{Theorem S1 action on ADE}
The $S^1-$action $\lambda\cdot (a,b)=(\lambda a,\lambda b)$ on
$\C^2/\Gamma$ lifts to a unique $I-$holomorphic $S^1-$action on
$(X,\omega_I)$. Moreover the $S^1-$action preserves the contact
hypersurface $\pi^{-1}(S_r^3/\Gamma)$ inside $(X,\omega_I)$
described in Lemma \ref{Lemma ADE contact type}, and the contact
form $\theta_I$ can be chosen to be $S^1-$equivariant.
\end{theorem}
\begin{proof}
Since $\Gamma$ is a complex group, it commutes with the diagonal
$S^1-$action on $\C^2$, therefore the action is well-defined on
$\C^2/\Gamma$. The lift of the action is
$$\varphi_{\lambda} (\alpha,\beta) = (\lambda \alpha,\lambda \beta).$$

In particular, the $S^1-$action preserves $\omega_I$ because it
preserves the metric $g$ and it commutes with the action of $I$.

Let $\theta_I$ denote the contact form constructed in Lemma
\ref{Lemma ADE contact type} for the hypersurface
$\pi^{-1}(S_r^3/\Gamma)$ and the symplectic form $\omega_I$. To
make $\theta_I$ an $S^1-$equivariant contact form, we simply
replace it by the $S^1-$averaged form
$\overline{\theta}_I=\int_0^1 \varphi_{e^{2\pi i t}}^* \theta_I \,
dt$. Since $\varphi_{\lambda}^*\omega_I=\omega_I$, it satisfies
$d\overline{\theta}_I=\omega_I$ and the positivity condition
$$\textstyle \overline{\theta}_I \wedge d\overline{\theta}_I
= \left(\int_0^1 \varphi_{e^{2\pi i t}}^* \theta_I \, dt\right)
\wedge \omega_I = \int_0^1 \varphi_{e^{2\pi i t}}^* (\theta_I
\wedge \omega_I) \, dt >0. \qedhere$$
\end{proof}

\begin{remark}
The $S^1-$action does not preserve $\omega_J$ and $\omega_K$. That
is why the symplectic cohomology for $\omega_I$ will be very
different from the one for $\omega_J$ or $\omega_K$.
\end{remark}
%
%
\subsection{Changing the contact hypersurface to a standard $S^3/\Gamma$}
\label{Subsection Changing the contact hypersurface to a standard
S3}
%
Our aim is to change the contact hypersurface in $(X,\omega_I)$ so
that it becomes a standard $S_r^3/\Gamma$. We want to do this
compatibly with the $S^1-$actions on $X$ and $\C^2/\Gamma$, so
that the $S^1-$action on $(X,\omega_I)$ will coincide with the new
Reeb flow. To do this, we need an $S^1-$equivariant version of
Gray's stability theorem.

\begin{lemma}[$S^1-$equivariant Gray stability]\label{Lemma Gray Stability}
For $t\in [0,1]$, let $\xi_t = \ker \alpha_t$ be a smooth family
of contact structures on some closed manifold $N^{2n-1}$. Then
there is an isotopy $\psi_t$ of $N$ and a family of smooth
functions $f_t$ such that
$$
\psi_t^*\alpha_t = e^{f_t} \alpha_0.
$$
If there is an $S^1-$action on $N$ preserving each $\alpha_t$,
then $f_t$ and $\psi_t$ are $S^1-$equivariant.
\end{lemma}
\begin{proof}
Let $X_t$ be a vector field inducing a flow $\psi_t$. By Cartan's
formula,
$$\partial_t \psi_t^*\alpha_t = \psi_t^*(\dot{\alpha}_t+
\mathcal{L}_{X_t}\alpha_t) = \psi_t^*(\dot{\alpha}_t+
di_{X_t}\alpha_t + i_{X_t}d\alpha_t).$$
Observe now that if $\psi_t$ satisfied the claim, then $\partial_t
\psi_t^*\alpha_t = \dot{f}_t
e^{f_t}\alpha_0=\psi_t^*(\dot{f}_t(\psi_t^{-1})\, \alpha_t)$.

We can reverse the argument to obtain the required $\psi_t$ if we
can find a vector field $X_t$ in $\xi_t$ (so $i_{X_t}\alpha_t=0$)
satisfying the equation
$$
i_{X_t}d\alpha_t =\dot{f}_t(\psi_t^{-1})\, \alpha_t -
\dot{\alpha}_t.
$$
Inserting the Reeb vector field $\mathcal{R}_t$ we obtain
$0=\dot{f}_t(\psi_t^{-1}) - \dot{\alpha}_t(\mathcal{R}_t)$.
Solving the latter equation determines $f_t$ with $f_0=0$. Then
the original equation determines $X_t\in \xi_t$ since $d\alpha_t$
is non-degenerate on $\xi_t$.

Suppose we had an $S^1-$action $\varphi_{\lambda}$ preserving
$\alpha$, $\varphi_{\lambda}^*\alpha_t = \alpha_t$. Applying
$\varphi_{\lambda}^*$ to the equation which determines $X_t$ at
$x$ we obtain the equation
$$
i_{\varphi_{\lambda}^* X_t}d\alpha_t =\dot{f}_t(\psi_t^{-1})\,
\alpha_t - \dot{\alpha}_t
$$
at $y=\varphi_{\lambda}^{-1}(x)$. The solution $f_t$ does not
change and so by uniqueness and $\varphi_{\lambda}^* X_t = X_t$,
which proves that $f_t$ and $\psi_t$ are $S^1-$equivariant.
\end{proof}

\begin{lemma} \label{Lemma Standard S3 in ADE}
The contact hypersurface $\pi^{-1}(S_r^3/\Gamma)$ can be deformed
inside $(X,\omega_I)$ into a copy of the standard $S_r^3/\Gamma$
via an $S^1-$equivariant contactomorphism.
\end{lemma}
\begin{proof}
Consider $X_t = X_{t\zeta_1,0,0}$ and denote by $\omega_t$ its
form $\omega_I$, ($0\leq t \leq 1$), and let $\pi_t: X_t \to
X_0=\mathbb{H}/\Gamma$ denote the minimal resolution. By Lemma
\ref{Lemma ADE contact type}, each $X_t$ comes with an
$S^1-$equivariant contact form $\theta_t$ with
$d\theta_t=\omega_t$ and such that $\theta_0$ is the standard
contact form on $S_r^3/\Gamma \subset X_0$. This defines a family
of $S^1-$equivariant contact forms $\alpha_t=(\pi_t)_*\theta_t$ on
$S_r^3/\Gamma$. By Lemma \ref{Lemma Gray Stability} there is an
$S^1-$equivariant isomorphism $(S_r^3/\Gamma,e^{f_1}\alpha_1) \to
(S_r^3/\Gamma,\alpha_0)$. In particular, this proves that $X$
arises by attaching an infinite collar to the manifold
$$\{ (R,x):
R\leq e^{f_1(\pi x)}, x\in \pi^{-1}(S_r^3/\Gamma)\} \subset X$$
along the boundary $\{ (e^{f_1(\pi x)},x)\} \subset X$ which is a
standard contact $S_r^3/\Gamma$.
\end{proof}

%
\subsection{Non-vanishing of the exact symplectic cohomology}
\label{Subsection Non-vanishing of the exact symplectic
cohomology}
%

\begin{theorem} $SH^*(X,\omega_u) \neq 0$ for $u=(0,u_J,u_K) \in S^2$,
indeed $c_*: H^*(X)\otimes \Lambda\to SH^*(X,\omega_u)$ is an
injection.
\end{theorem}
\begin{proof}
The exceptional spheres in $X$ are exact by Lemma \ref{Lemma
exceptional divisors are Lagr}. For each such $S^2$ we have a
commuting diagram by Theorem \ref{Theorem Viterbo Functoriality},
using the bundles described in \ref{SubsectionNovikovbundles}:
$$
\xymatrix{ H_{4-*}(\mathcal{L}S^2)\otimes \Lambda\cong
SH^*(T^*S^2,d\theta) \ar@{<-}[r] \ar@{<-}@<-10ex>[d]^-{c_*} &
SH^*(X,\omega_u) \ar@{<-}[d]^-{c_*} \\
H_{4-*}(S^2)\otimes \Lambda \cong H^*(S^2)\otimes \Lambda
\ar@{<-}[r]^-{i^*} & H^{*}(X)\otimes \Lambda }
$$

The left vertical map is induced by the inclusion of constant
loops and it is injective on homology because it has a left
inverse by evaluation at $0$. Since $H^*(X)$ is generated by the
exceptional spheres by Theorem \ref{Theorem Kronheimer}, and $i^*$
is the projection to the summands of $H^*(X)$, the claim follows.
\end{proof}

%
\subsection{Vanishing of the non-exact symplectic cohomology}
\label{Subsection Vanishing of the non-exact symplectic
cohomology}
%
\begin{theorem}\label{Theorem vanishing of nonexact SH of ADE}
$SH^*(X,\omega_I)=0$.
\end{theorem}
\begin{proof}
By Theorem \ref{Theorem Invariance under Contactomorphs} the
symplectic cohomology changes by an isomorphism if we choose a
different contact hypersurface in the collar. By Lemma \ref{Lemma
Standard S3 in ADE}, we changed the hypersurface by an
$S^1-$equivariant contactomorphism so that the collar of $X$
(after metric rescaling) can be assumed to be the standard
$S^3/\Gamma \times [1,\infty)$ with $S^1-$action $(a,b)\mapsto
(\lambda a, \lambda b)$. The symplectic $S^1-$action
$\varphi_{\lambda}$ on $(X,\omega_I)$ defines a vector field
$$
X_{\varphi}(x) = \left.\frac{\partial}{\partial t}\right|_{t=0}
\varphi_{e^{2\pi i t}}(x).
$$
By Cartan's formula $0=\partial_{\lambda}\varphi_{\lambda}^*\omega
= \varphi_{\lambda}^* \mathcal{L}_{X_{\varphi}} \omega =
di_{X_{\varphi}}\omega$. Thus, since $H^1(X;\R)=0$, we obtain a
Hamiltonian via $i_{X_{\varphi}} \omega = -dH_{\varphi}$.
Moreover, accelerating the flow by a factor $k$, we obtain an
$S^1-$action $\varphi_{k\lambda}$ with Hamiltonian $H_k=
kH_{\varphi}$. On the collar, $H_k(a,b)=k \pi (|a|^2+|b|^2)$ and
since $R=|a|^2+|b|^2$, the Hamiltonian is linear at infinity:
$h_k(R)=k\pi R$.

The $1-$periodic orbits of $H_k$ either lie in $\pi^{-1}(0)$ or
come from lifts of nonconstant $1-$periodic orbits on
$\C^2/\Gamma$ for the flow $(a,b)\mapsto (\lambda^k a,\lambda^k
b)$. But for generic $k$, there are no $1-$periodic orbits of
$H_k$ on $\C^2/\Gamma$ except for $0$. So we reduce to calculating
the Maslov indices of $1-$periodic orbits in $\pi^{-1}(0)$.

Since the flow $\varphi_{\lambda}$ is holomorphic, the
linearization over a periodic orbit will be a loop of unitary
transformations. Its Maslov index can therefore be calculated as
the winding number of the determinant of the linearization in the
trivialization $\C\cdot (\omega_J + i \omega_K)$ of the canonical
bundle. Since
$$
\varphi_{\lambda}^*\omega_J(V,W) = g(J\varphi_{\lambda *} V,
\varphi_{\lambda *} W) = g(J \lambda V, \lambda W) = \lambda^2
\omega_J(V,W),
$$
and similarly for $K$, we deduce that $\varphi_{\lambda}$ acts on
the canonical bundle by rotation by $\lambda^2$. Therefore the
Maslov index increases by $2$ for each full rotation of $\lambda$.

We deduce that the Maslov indices for $H_k$ grow to infinity as
$k\to \infty$. Therefore the generators of
$SH^*(H_{k+N},\omega_I)$ have arbitrarily negative Conley-Zehnder
indices as $N\to \infty$, and so the image of $SH^m(H_k,\omega_I)$
under the continuation map
$$SH^m(H_k,\omega_I)\to SH^m(H_{k+N},\omega_I)$$
vanishes for large $N$. Thus the direct limit $SH^m(X,\omega_I)=0$
for all $m$.
\end{proof}

\begin{corollary}
\label{Corollary vanishing of generic SH of ADE plumbings} Let $X$
be an ALE space. Given a generic class in $H^2(X;\R)$, it is
possible to choose a symplectic form $\omega$ on $X$ representing
that class such that
$$
SH^*(X,\omega) =0.$$
Here, genericity refers to choosing $[\omega]$ in the complement
of certain finitely many hyperplanes in $H^2(X;\R)$.
\end{corollary}
\begin{proof}
All the $X_{\zeta_1,0,0}$ for generic $\zeta_1$ are diffeomorphic
(Theorem \ref{Theorem Kronheimer}.1). We fix one such choice
$X=X_{a,0,0}$, and we consider the family of forms $\omega_I$
induced on $X$ by pull-back from $X_{\zeta_1,0,0}$ via the
diffeomorphism $X\cong X_{\zeta_1,0,0}$. By Lemma \ref{Lemma
exceptional divisors are Lagr}, $[\omega_I]$ will range over all
generic choices in $H^2(X;\R)$ (genericity of $\omega_I$
corresponds to the genericity of $\zeta_1$). The result now
follows from Theorem \ref{Theorem vanishing of nonexact SH of
ADE}.
\end{proof}

%
\subsection{Vanishing of the twisted symplectic cohomology}
\label{Subsection Vanishing of the twisted symplectic cohomology}
%
\begin{lemma} \label{Lemma SH infinitesimal same as nonexact ADE}
The non-compactly supported deformation from $\omega_J$ to
$\omega_I$ can be made to satisfy Theorem \ref{Theorem Lim of tau
twisting maps}, thus $
SH^*(X,\omega_J;\underline{\Lambda}_{\tau\omega_I}) \cong
SH^*(X,\omega_I).$
\end{lemma}
\begin{proof}
Let $\omega_{\varepsilon}=\omega_J+\varepsilon \omega_I$. By the
proof of Lemma \ref{Lemma ADE contact type}, we can find a family
of contact forms $\theta_{\varepsilon}|_S$ on
$S=\pi^{-1}(S_r^3/\Gamma)$ with
$d\theta_{\varepsilon}=\omega_{\varepsilon}$. By Gray's stability
theorem, there is a family of contactomorphisms
$\psi_{\varepsilon}:S\to S$ such that
$\psi_{\varepsilon}^*(\theta_{0}|_S) =
e^{f_{\varepsilon}}\theta_{\varepsilon}|_S$. As we deform
$\omega_0$ to $\omega_{\varepsilon}$ we simultaneously change the
hypersurface in $X$ by
$$
S \to X, (R,x)\mapsto
(e^{-f_{\varepsilon}(R,x)}R,\psi_{\varepsilon}(R,x)),
$$
so that on the collar determined by this hypersurface the one-form
is $\theta_0$ instead of $\theta_{\varepsilon}$. This change of
hypersurface will change the symplectic cohomology by an
isomorphism (Theorem \ref{Theorem Invariance under
Contactomorphs}). The ``interior part" of $X$ has changed by a
diffeomorphism, and we have reduced the setup to the case where we
deform $\omega_0$ to a form $\omega_{\varepsilon}'$ which is
cohomologous to $\omega_{\varepsilon}$ but which equals
$d\theta_0$ on the collar.

Now it is possible to make a compactly supported deformation from
$\omega_J$ to $\omega_{\varepsilon}'$ and, for small
$\varepsilon$, Theorem \ref{Theorem Lim of tau twisting maps}
implies that
$ SH^*(X,\omega_J;\underline{\Lambda}_{\varepsilon\tau\omega_I})
\cong SH^*(X,\omega_{\varepsilon}')$. Rescale by $1/\varepsilon$
via $t \mapsto t^{1/\varepsilon}$ to deduce that
$$
SH^*(X,\omega_J;\underline{\Lambda}_{\tau\omega_I}) \cong
SH^*(X,\omega_{\varepsilon}'/{\varepsilon}).
$$
Now $\omega_{\varepsilon}'/{\varepsilon}$ is cohomologous to
$\omega_I$. By applying Gray's theorem as above, we can change
$\omega_{\varepsilon}'/{\varepsilon}$ within its cohomology class
so that on the collar it becomes equal to $\omega_I$. Finally we
apply a compactly-supported Moser symplectomorphism as in Lemma
\ref{Lemma Moser deformation} to deform the form to $\omega_I$ on
all of $X$. Hence
$$
SH^*(X,\omega_{\varepsilon}'/{\varepsilon}) \cong
SH^*(X,\omega_I).\qedhere
$$
\end{proof}

\begin{theorem}
Let $X$ be an ALE space and let $d\theta$ denote any non-zero
linear combination of $\omega_J$ and $\omega_K$. For any generic
closed two-form $\beta$ on $X$,
$$
SH^*(X,d\theta;\Lambda_{\tau\beta}) =0.
$$
Again, generic is understood in the sense of Corollary
\ref{Corollary vanishing of generic SH of ADE plumbings}.
\end{theorem}
\begin{proof} By the proof of
Corollary \ref{Corollary vanishing of generic SH of ADE plumbings}
we may assume that $\omega_I$ represents $[\beta]$. In particular,
$SH^*(X,\omega_I)=0$. Note that a non-zero linear combination of
$\omega_J$ and $\omega_K$ is of the form $c\, \omega_u$ for some
$c>0$ and $u=(0,u_J,u_K)\in S^2$, and is therefore exact. The
proof of Lemma \ref{Lemma SH infinitesimal same as nonexact ADE}
can easily be adapted to $d\theta=c\,\omega_u$, so
$SH^*(X,d\theta;\underline{\Lambda}_{\tau \omega_I})=0$. By Lemma
\ref{Lemma Moser deformation twisted}, $\Lambda_{\tau \omega_I}$
only depends on the cohomology class of $\omega_I$ up to
isomorphism, therefore $SH^*(X,d\theta;\Lambda_{\tau\beta})=0$.
\end{proof}
%
\subsection{Exact Lagrangians in ALE spaces}
\label{Subsection Exact Lagrangians in ADE spaces}
%
%
\begin{theorem}
Let $X$ be an ALE space. Then any exact Lagrangian submanifold $j:
L \hookrightarrow (X,\omega_J)$ must be a sphere, in particular
$L$ cannot be unorientable. This result also holds if we replace
$\omega_J$ by any non-zero combination of $\omega_J$ and
$\omega_K$.
\end{theorem}
\begin{proof}
Since $SH^*(X,\omega_J;\underline{\Lambda}_{\tau\omega_I})=0$,
Corollary \ref{Corollary Functoriality trick} implies that the
transgression $\tau(j^*[\omega_I])$ cannot vanish. But for
orientable $L$ which are not spheres all transgressions must
vanish since $\pi_2(L)=0$. Therefore the only allowable orientable
exact Lagrangians are spheres. The unorientable case follows by
Remark \ref{Remark Unorientable exact Lagrangians}.
\end{proof}
\begin{corollary}
Let $(Y,d\theta)$ be the plumbing of copies of $T^*S^2$ as
prescribed by any $ADE$ Dynkin diagram. Then any exact Lagrangian
$L \subset Y$ must be a sphere, in particular $L$ cannot be
unorientable.
\end{corollary}
\begin{proof}
This follows immediately by Section \ref{Subsection Plumbing
construction of ADE spaces} (or, as mentioned in Remark
\ref{Remark Plumbing or ALE does not matter}, by embedding $Y$
into an ALE space $X$ with the same Dynkin diagram).\end{proof}
%
%
%
%
%
%
%
%

\end{document}

%% file: a_dynkin.pstex_t
\begin{picture}(0,0)%
\includegraphics{a_dynkin.pstex}%
\end{picture}%
\setlength{\unitlength}{2960sp}%
\begingroup\makeatletter\ifx\SetFigFont\undefined%
\gdef\SetFigFont#1#2#3#4#5{%
  \reset@font\fontsize{#1}{#2pt}%
  \fontfamily{#3}\fontseries{#4}\fontshape{#5}%
  \selectfont}%
\fi\endgroup%
\begin{picture}(5965,1812)(939,-2541)
\put(2234,-2000){\makebox(0,0)[lb]{\smash{{\SetFigFont{8}{9.6}{\rmdefault}{\mddefault}{\updefault}$(n=7)$}}}}
\put(2236,-1058){\makebox(0,0)[lb]{\smash{{\SetFigFont{8}{9.6}{\rmdefault}{\mddefault}{\updefault}$(n=6)$}}}}
\put(1879,-1072){\makebox(0,0)[lb]{\smash{{\SetFigFont{9}{10.8}{\rmdefault}{\mddefault}{\updefault}$A$}}}}
\put(5078,-984){\makebox(0,0)[lb]{\smash{{\SetFigFont{9}{10.8}{\rmdefault}{\mddefault}{\updefault}$E$}}}}
\put(5087,-1613){\makebox(0,0)[lb]{\smash{{\SetFigFont{9}{10.8}{\rmdefault}{\mddefault}{\updefault}$E$}}}}
\put(5089,-2243){\makebox(0,0)[lb]{\smash{{\SetFigFont{9}{10.8}{\rmdefault}{\mddefault}{\updefault}$E$}}}}
\put(1897,-2021){\makebox(0,0)[lb]{\smash{{\SetFigFont{9}{10.8}{\rmdefault}{\mddefault}{\updefault}$D$}}}}
\put(2041,-1137){\makebox(0,0)[lb]{\smash{{\SetFigFont{8}{9.6}{\rmdefault}{\mddefault}{\updefault}$n$}}}}
\put(2054,-2086){\makebox(0,0)[lb]{\smash{{\SetFigFont{8}{9.6}{\rmdefault}{\mddefault}{\updefault}$n$}}}}
\put(5235,-1066){\makebox(0,0)[lb]{\smash{{\SetFigFont{8}{9.6}{\rmdefault}{\mddefault}{\updefault}$6$}}}}
\put(5254,-1682){\makebox(0,0)[lb]{\smash{{\SetFigFont{8}{9.6}{\rmdefault}{\mddefault}{\updefault}$7$}}}}
\put(5248,-2317){\makebox(0,0)[lb]{\smash{{\SetFigFont{8}{9.6}{\rmdefault}{\mddefault}{\updefault}$8$}}}}
\end{picture}%

%% file: compactness.pstex_t
\begin{picture}(0,0)%
\includegraphics{compactness.pstex}%
\end{picture}%
\setlength{\unitlength}{3947sp}%
\begingroup\makeatletter\ifx\SetFigFont\undefined%
\gdef\SetFigFont#1#2#3#4#5{%
  \reset@font\fontsize{#1}{#2pt}%
  \fontfamily{#3}\fontseries{#4}\fontshape{#5}%
  \selectfont}%
\fi\endgroup%
\begin{picture}(3792,4866)(1126,-4400)
\put(1814,-724){\makebox(0,0)[lb]{\smash{{\SetFigFont{17}{20.4}{\rmdefault}{\mddefault}{\itdefault}$u'$}}}}
\put(2826,226){\makebox(0,0)[lb]{\smash{{\SetFigFont{17}{20.4}{\rmdefault}{\mddefault}{\itdefault}$x$}}}}
\put(3739,-749){\makebox(0,0)[lb]{\smash{{\SetFigFont{17}{20.4}{\rmdefault}{\mddefault}{\itdefault}$u''$}}}}
\put(4364,-2099){\makebox(0,0)[lb]{\smash{{\SetFigFont{17}{20.4}{\rmdefault}{\mddefault}{\itdefault}$y''$}}}}
\put(3751,-3074){\makebox(0,0)[lb]{\smash{{\SetFigFont{17}{20.4}{\rmdefault}{\mddefault}{\itdefault}$u''$}}}}
\put(2339,-2361){\makebox(0,0)[lb]{\smash{{\SetFigFont{17}{20.4}{\rmdefault}{\mddefault}{\itdefault}$u$}}}}
\put(2814,-4311){\makebox(0,0)[lb]{\smash{{\SetFigFont{17}{20.4}{\rmdefault}{\mddefault}{\itdefault}$y$}}}}
\put(1689,-3136){\makebox(0,0)[lb]{\smash{{\SetFigFont{17}{20.4}{\rmdefault}{\mddefault}{\itdefault}$u'$}}}}
\put(1839,-3211){\makebox(0,0)[lb]{\smash{{\SetFigFont{10}{12.0}{\rmdefault}{\bfdefault}{\updefault}$2$}}}}
\put(3901,-3161){\makebox(0,0)[lb]{\smash{{\SetFigFont{10}{12.0}{\rmdefault}{\bfdefault}{\updefault}$2$}}}}
\put(3876,-836){\makebox(0,0)[lb]{\smash{{\SetFigFont{10}{12.0}{\rmdefault}{\bfdefault}{\updefault}$1$}}}}
\put(1951,-811){\makebox(0,0)[lb]{\smash{{\SetFigFont{10}{12.0}{\rmdefault}{\bfdefault}{\updefault}$1$}}}}
\put(1126,-2099){\makebox(0,0)[lb]{\smash{{\SetFigFont{17}{20.4}{\rmdefault}{\mddefault}{\itdefault}$y'$}}}}
\put(2489,-2436){\makebox(0,0)[lb]{\smash{{\SetFigFont{12}{14.4}{\rmdefault}{\bfdefault}{\updefault}$n$}}}}
\end{picture}%

%% file: viterbo1.pstex_t
\begin{picture}(0,0)%
\includegraphics{viterbo1.pstex}%
\end{picture}%
\setlength{\unitlength}{3947sp}%
\begingroup\makeatletter\ifx\SetFigFont\undefined%
\gdef\SetFigFont#1#2#3#4#5{%
  \reset@font\fontsize{#1}{#2pt}%
  \fontfamily{#3}\fontseries{#4}\fontshape{#5}%
  \selectfont}%
\fi\endgroup%
\begin{picture}(5811,3728)(689,-3737)
\put(1213,-2448){\makebox(0,0)[lb]{\smash{{\SetFigFont{14}{16.8}{\rmdefault}{\mddefault}{\itdefault}$c$}}}}
\put(1738,-2473){\makebox(0,0)[lb]{\smash{{\SetFigFont{14}{16.8}{\rmdefault}{\mddefault}{\itdefault}$1$}}}}
\put(5310,-240){\makebox(0,0)[lb]{\smash{{\SetFigFont{14}{16.8}{\rmdefault}{\mddefault}{\itdefault}$b$}}}}
\put(4270,-2529){\makebox(0,0)[lb]{\smash{{\SetFigFont{14}{16.8}{\rmdefault}{\mddefault}{\itdefault}$\partial M$}}}}
\put(2214,-2472){\makebox(0,0)[lb]{\smash{{\SetFigFont{14}{16.8}{\rmdefault}{\mddefault}{\itdefault}$1+\varepsilon$}}}}
\put(949,-525){\makebox(0,0)[lb]{\smash{{\SetFigFont{14}{16.8}{\rmdefault}{\mddefault}{\itdefault}$h$}}}}
\put(1515,-1257){\makebox(0,0)[lb]{\smash{{\SetFigFont{14}{16.8}{\rmdefault}{\mddefault}{\itdefault}$a$}}}}
\put(1576,-2111){\makebox(0,0)[lb]{\smash{{\SetFigFont{14}{16.8}{\rmdefault}{\mddefault}{\itdefault}$\partial W$}}}}
\put(6039,-2511){\makebox(0,0)[lb]{\smash{{\SetFigFont{14}{16.8}{\rmdefault}{\mddefault}{\itdefault}$R$}}}}
\end{picture}%